\documentclass[a4paper,10pt]{amsart}

\usepackage[english]{babel}
\usepackage{amsmath,amssymb}

\usepackage[centering]{geometry}
\usepackage{color}

\usepackage{relsize} 

\usepackage{cite}

\usepackage{amsmath,amssymb,amsthm}

\usepackage{tikz-cd}

\usepackage{hyperref}
\usepackage[noabbrev,capitalise]{cleveref}

\newtheorem{theorem}{Theorem}[section] 
\newtheorem{lemma}[theorem]{Lemma}
\newtheorem{proposition}[theorem]{Proposition}

\theoremstyle{definition}
\newtheorem{definition}[theorem]{Definition}
\newtheorem{corollary}[theorem]{Corollary}
\newtheorem{remark}{Remark}

\usepackage{thmtools}
\declaretheoremstyle[
    spaceabove=6pt, spacebelow=6pt,
    headfont=\normalfont\bfseries,
    notefont=\mdseries, notebraces={(}{)},
    bodyfont=\normalfont,
    postheadspace=1em,
    qed=\ensuremath{\bigcirc} 
]{examplestyle}

\declaretheorem[style=examplestyle, name=Example]{example}

\numberwithin{equation}{section}

\usepackage{multicol}
\usepackage{mathtools}

\let\svqty\qty
\usepackage{physics}
\let\qty\svqty

\usepackage[shortlabels]{enumitem}
\usepackage[numbered]{bookmark}

\usepackage{float}
\usepackage{graphicx,xcolor}
\usepackage{grffile}
\usepackage[labelfont=bf, labelsep=space]{caption} 
\usepackage[labelfont=default]{subcaption}

\usepackage{dsfont} 

\usepackage{upgreek} 

\newcommand{\hpi}{\uppi} 

\usepackage{csquotes}
\DeclareQuoteAlias[american]{english}{dutch}
\usepackage{todonotes}

\DeclareMathOperator{\vct}{span}

\usepackage{arydshln} 
\usepackage{multirow}

\newcommand{\inprod}[2]{ \langle \, #1, \, #2 \, \rangle }

\newcommand*{\integers}{\mathbb Z}
\newcommand*{\reals}{\mathbb R}
\newcommand*{\complexs}{\mathbb{C}}

\newcommand*{\set}[1]{\l\{ #1 \r\}} 
\renewcommand{\l}{\left}
\renewcommand{\r}{\right}

\renewcommand*{\epsilon}{\varepsilon}

\renewcommand*{\phi}{\varphi}

\newcommand{\identity}{\mathds{1}} 
\newcommand{\CP}{{\complexs P}}
\newcommand{\sphere}[1]{\mathbb{S}^{#1}}
\newcommand{\quaternions}{\mathbb{H}}

\newcommand{\J}{{J}}

\newcommand{\lift}[1]{\widetilde{#1}}

\newcommand{\Dtwo}{\mathcal{D}_1^2}
\newcommand{\Dfour}{\mathcal{D}_2^4}
\newcommand{\liftDtwo}{\lift{\mathcal{D}}_1^2}
\newcommand{\liftDfour}{\lift{\mathcal{D}}_2^4}

\newcommand{\FS}[1]{#1_{1} }
\newcommand{\FSm}{g_{1}}
\newcommand{\FSnabla}{\nabla^{1}}

\newcommand{\cyclic}[1]{\underset{#1}{\mathlarger{\mathlarger{\mathlarger{\mathfrak{S}}}}}}

\newcommand{\Jo}{J_1}

\newcommand{\metricparam}{a}

\DeclareMathOperator{\Ric}{Ric}

\DeclareMathOperator{\SU}{SU}
\DeclareMathOperator{\U}{U}
\DeclareMathOperator{\Sp}{Sp}

\DeclareMathOperator{\Iso}{Iso}

\newcommand\su{\mathfrak{su}}
\newcommand\so{\mathfrak{so}}
\renewcommand\u{\mathfrak{u}}
\renewcommand\sp{\mathfrak{sp}}

\DeclareMathOperator{\Span}{span}

\usepackage{tikz-cd}

\makeatletter
\renewcommand\paragraph{\@startsection{paragraph}{4}{\z@}%
{-2.5ex\@plus -1ex \@minus -.25ex}%
{1.25ex \@plus .25ex}%
{\normalfont\normalsize\bfseries}}
\makeatother
\expandafter\def\expandafter\normalsize\expandafter{%
   \normalsize%
   \setlength\abovedisplayskip{4pt}%
   \setlength\belowdisplayskip{4pt}%
   \setlength\abovedisplayshortskip{2pt}%
   \setlength\belowdisplayshortskip{2pt}%
}

\title[Curvature and Lagrangian submanifolds of nearly Kähler $\CP^3$]{Curvature and Lagrangian submanifolds of \\ the homogeneous nearly Kähler $\CP^3$}

\author{Micha\"el Liefsoens \and Joeri Van der Veken}

\address{M. Liefsoens \and J. Van der Veken, KU\ Leuven, Department of Mathematics, Celestij\-nenlaan 200B -- Box 2400, 3001 Leuven, Belgium}
\email{michael.liefsoens@kuleuven.be}
\email{joeri.vanderveken@kuleuven.be}
\thanks{M. Liefsoens is supported by the Research Foundation--Flanders (FWO) under project 11PG324N. J. Van der Veken is supported by the KU Leuven Research Fund under project 3E210539, by the Research Foundation--Flanders (FWO) and the Fonds de la Recherche Scientifique (FNRS) under EOS Project G0I2222N, and by the Research Foundation--Flanders (FWO) under mobility project VS05726N. }


\date{}

\begin{document}

\begin{abstract}
A tractable definition of the homogeneous nearly Kähler structure on $\CP^3$ is given via the Hopf fibration, facilitating explicit computations and analysis. The description extends to all homogeneous metrics on $\CP^3$, providing expressions for their Riemann curvature tensors and full isometry groups. Rigid immersions are presented for all extrinsically homogeneous Lagrangian submanifolds in the nearly Kähler $\CP^3$, and the nonexistence of Lagrangians with constant sectional curvature is established. \vspace{-.5cm}
\end{abstract}

\maketitle

\section{Introduction}

The definition of a nearly Kähler manifold dates back to Tachibana \cite{tachibana1959} in 1959 and goes as follows: it is a Hermitian manifold $(M,g,J)$ such that $\nabla J$ is skew-symmetric. Here, $J$ is the almost complex structure and $\nabla$ is the Levi-Civita connection of $g$. If, in addition, $\nabla J$ is non-degenerate, $(M,g,J)$ is called a strict nearly Kähler manifold. Nearly Kähler manifolds are natural and fascinating generalizations of Kähler manifolds (where $\nabla J = 0$). They have important applications in physics, see for example \cite{atiyah2002} and \cite{grunewald1985}; and are fundamental in $G_2$-geometry, see for example \cite{foscolo2017}. The concept of nearly Kähler manifold was later studied in great detail by Gray in a series of papers \cite{gray1970,gray1966,gray1976}, where he also introduced the name ``nearly Kähler''.

The six-dimensional nearly Kähler manifolds are key in the general theory, as is evidenced by the following fundamental results. First, six-dimensional nearly Kähler manifolds are the lowest-dimensional non-Kähler examples. Second, Nagy's structure theorem \cite{nagy2002} states that any complete strict nearly Kähler manifold is a Riemannian product with factors that are (1)~six-dimensional manifolds, (2)~twistor spaces over quaternionic manifolds, or (3)~homogeneous nearly Kähler spaces. Butruille \cite{butruille2005} showed that, up to homotheties, there are only four homogeneous six-dimensional nearly Kähler manifolds: the round six-sphere $\sphere{6}$, the product of three-spheres $\sphere{3}\times \sphere{3}$, the space $F_{1,2}$ of full flags in $\complexs^3$ and the complex projective space $\CP^3$. Note that the latter two are unique in that they fall in each category of Nagy's structure theorem; that the metric on $\sphere{3}\times \sphere{3}$ is not a Riemannian product of round metrics; and that the metric of $\CP^3$ is not the Fubini-Study  metric. In addition, Foscolo and Haskins \cite{foscolo2017} discovered  the first two complete six-dimensional non-homogeneous nearly Kähler manifolds. 

The nearly Kähler $\CP^3$ is a twistor space over $\sphere{4}$ and it is a homogeneous space. In particular, there is a Riemannian submersion $\CP^3 \to \sphere{4}$ and $\CP^3$ can be written as $\Sp(2)/(\Sp(1) \U(1))$. The first description has the advantage to link $\CP^3$ into the general theory of twistor spaces, but has the drawback that $\CP^3$ is the domain of a submersion, which makes it harder to work with. The second description is extremely useful to classify equivariant submanifolds, but is harder to work with for general immersions. To overcome these difficulties, we instead propose an alternative description: $\CP^3$ as the base space of the Hopf fibration $\hpi: \sphere{7} \to \CP^3$, with a squashed metric on $\sphere{7}$.

To understand nearly Kähler spaces better, it is essential to study their submanifolds, and in particular those that interact naturally with the nearly Kähler structure. In the nearly Kähler $\CP^3$, submanifolds have been studied the least, when comparing to the other three six-dimensional homogeneous examples. 

Lagrangian submanifolds are deeply related to the nearly Kähler structure as the latter induces an isometry between the tangent space and the normal space at any point of the submanifold. The Lagrangian submanifolds of the nearly Kähler $\sphere{6}$ are well-studied in (among others) references \cite{dillen1990,dillen1996,ejiri1981,vrancken1998,enoyoshi2020,hu2020,hu2019}, just as they are well-studied in the nearly Kähler $\sphere{3}\times\sphere{3}$, see for example \cite{bektas2019,dioos2018}. For $F_{1,2}$ and $\CP^3$, the submanifolds that are Lagrangian for both the Kähler and the nearly Kähler structure have been studied from a twistor fibration point of view in \cite{storm2020}. In addition, Aslan classified all totally geodesic Lagrangians of $\CP^3$ and those invariant under an $\SU(2)$-subgroup of the isometry group \cite{aslan2023}. 

This article aims to first simplify the existing descriptions of the nearly Kähler $\CP^3$, which generalizes to all homogeneous metrics on $\CP^3$. Moreover, it leads to explicit expressions for the curvature tensors (\Cref{thm:curvature_CP3a}) and to descriptions of the full isometry groups (\Cref{thm:isometries_NK_characterised}) of all these metrics. Secondly, several results on Lagrangian submanifolds of the homogeneous nearly Kähler $\CP^3$ are obtained. In particular, extrinsically homogeneous Lagrangians are classified in \Cref{thm:classification}, and explicit and rigid immersions for all these examples are given in  \Cref{eq:RP3,eq:Chiang,eq:berger,eq:S2xS1,eq:EHL}. Moreover, the non-existence of Lagrangians with constant sectional curvature is proven in \Cref{thm:cst_sec_curv}.

Most of the content of this article is based on the first author's MSc thesis \cite{liefsoens2024thesis} of the academic year 2021--2022.

\section{Preliminaries}

\subsection{Kähler \texorpdfstring{$\CP^3$}{ℂP³}}\label{sec:kahlerCP3}
The Hopf fibration $\hpi : \sphere{7} \subset \mathbb R^8 \to \CP^3$ is a Riemannian submersion when $\mathbb S^7$ is equipped with the round metric of constant curvature $1$ and $\CP^3$ with the Fubini-Study metric $\FSm$. Under the identification $\mathbb R^8 \cong \mathbb C^4$, the vertical distribution is given by $V(q) = \Span\{iq\}$. The induced almost complex structure on $\CP^3$ is denoted by $\Jo$, and turns $(\CP^3, \FSm, \Jo)$ into a Kähler manifold. The curvature tensor of this Kähler manifold is given by
\begin{equation}
\FS{R}(X,Y)Z = (X \FS{\wedge} Y)Z + (\Jo X \FS{\wedge} \Jo Y)Z +2 \FSm(X,\Jo  Y) \Jo Z,
\end{equation}
where we used the notation $(X \FS{\wedge} Y)Z = \FSm(Y,Z)X - \FSm(X,Z)Y$.

The Kähler $\CP^3$ has a complex almost contact structure, as described in \cite{blair2000}: there exist (local) unit vector fields $U$ and $V=\Jo U$, with dual one-forms $u$ and $v$ respectively, and (local) fields of endomorphisms $\Phi$ and $\Psi=\Jo \Phi$, such that, denoting $\sigma(X) = \FSm(\FSnabla_X U, V)$,
\begin{subequations} \label{eq:complex_almost_contact_structure}
\begin{align}
&\Phi U = \Phi V = 0,  &
&\Phi^2  = \Psi^2 = -\identity + u \otimes U + v \otimes V, \label{eq:psi_only_on_D24_and_square_psi}  \\ 
&\FSm(X,\Phi Y) = - \FSm(\Phi X,Y), &
&\Phi \Jo = -\Jo \Phi,  \label{eq:psi_antiCommutes_JO} \\
&\FSnabla_X U = -\Psi X + \sigma(X) V, &
&\FSnabla_X V = -\Phi X - \sigma(X) U \label{eq:nabla0_UV},
\end{align}
\vspace{-.55cm}
\begin{align}
\hspace{-.8cm} (\FSnabla_X \Psi) Y &= \FSm(X,Y) U - u(Y)X - \FSm(X,\Jo Y)V-v(Y)\Jo X + \sigma(X)\Phi Y \label{eq:nabla0_psi}, \\
(\FSnabla_X \Phi) Y &= \FSm(X,Y) V - v(Y)X - \FSm(X,\Jo Y)U+u(Y)\Jo X - \sigma(X)\Psi Y \label{eq:nabla0_phi}.
\end{align}
\end{subequations}
In addition, the horizontal lift of the vector field $U$ lies in the span of the vector fields $p \mapsto j p$ and $p \mapsto kp$, where we identify $\reals^8 \cong \complexs^4$ with $\quaternions^2$. 

\subsection{Nearly Kähler manifolds of constant type}
A nearly Kähler manifold is said to be of constant type $\alpha\in \reals$ if $g(G(X,Y),G(X,Y)) = \alpha \, ( g(X,X)g(Y,Y) - g(X,Y)^2 - g(X, J Y)^2 )$ for all tangent vectors $X$ and $Y$, where $G = \nabla J$, see for example \cite{gray1970}. Gray showed \cite{gray1976} that every strict nearly Kähler six-manifold is of constant type. For a nearly Kähler manifold of constant type $\alpha$, the following equations hold:
\begin{align*}
& (\nabla G)(X, Y,Z) = \alpha^2 ( J (Y\wedge Z)X  - g(JY,Z)X ), \\
& G(X,G(Y,Z)) = \alpha^2 ( (Y\wedge Z)X  + J (Y\wedge Z)JX ), \\ 
& g(G(X,Y),G(Z,W)) = \alpha^2 ( g\l( (Z\wedge W)Y, X\r) + g\l( J(Z\wedge W)JY, X\r) ).
\end{align*}

\subsection{A congruence result}

We will use the following result from \cite{dioos2018} to prove rigidity of the coming examples. 

\begin{proposition}[\cite{dioos2018}]\label{thm:constant_connection_coef_ONF_submanifold}
Let $(M, g)$ and $(\tilde{M}, \tilde{g})$ be $n$-dimensional Riemannian manifolds with Levi-Civita connections $\nabla$ and $\tilde{\nabla}$. Suppose that there exist constants $c_{ij}^k$ $(i,j,k\in\set{1,\hdots,n})$ such that for all $p\in M$ and $\tilde{p}\in \tilde{M}$ there exist local orthonormal frames $\set{E_1,\hdots,E_n}$ around $p$ and $\set{E'_1,\hdots,E'_n}$ around $\tilde{p}$ such that $g(\nabla_{E_i}E_j, E_k) = c_{ij}^k = \tilde{g}(\tilde{\nabla}_{E'_i}E'_j, E'_k)$. Then, for any $p\in M$ and $\tilde{p}\in\tilde{M}$ there exists a local isometry which maps a neighbourhood of $p$ to a neighbourhood of $\tilde{p}$ and $E_i$ to $E'_i$ for all $i \in \set{1, \hdots, n}$.    
\end{proposition}

\section{Nearly Kähler \texorpdfstring{$\CP^3$}{ℂP³} via the Hopf fibration}\label{sec:description}

\subsection{The almost product structure on \texorpdfstring{$\CP^3$}{ℂP³}} \label{sec:distributions}

Consider the Hopf fibration $\hpi: \mathbb S^7 \subset \mathbb R^8 \to~\CP^3$. After identifying $\mathbb R^8 \cong \mathbb C^4 \cong \mathbb H^2$, we can use multiplication by $i$, $j$ and $k$. As stated in the Preliminaries, the vertical distribution is given by $V(q) = \Span\{iq\}$. Note that the distribution $\liftDtwo(q) = \Span \set{j q, k q}$ is horizontal. We thus obtain an orthogonal splitting $T\sphere{7} = V \oplus \liftDtwo \oplus \liftDfour$, where $\liftDfour$ is given by $\liftDfour(q) = \Span\set{i q, j q, k q}^{\perp} \subset T_q\sphere{7}$. These distributions project well to $\CP^3$, as the following proposition shows.

\begin{proposition}\label{thm:decomposition_tangent_space_CP3}
Let $q\in\mathbb S^7$ and use the notations introduced above. If $p = \hpi(q)\in\CP^3$,  then the distributions $\Dtwo(p) = \dd \hpi_q( \liftDtwo(q) )$ and $\Dfour(p) = \dd \hpi_q ( \liftDfour(q) )$ are well defined and invariant under $\Jo$. Furthermore, $T_p(\CP^3)$ decomposes $\FSm$-orthogonally as $T_p(\CP^3) = \Dtwo(p) \oplus \Dfour(p)$.
\end{proposition} 

\begin{proof}
Note that $\FSm$-orthogonality is immediate, as $\FSm$ comes from the round metric on $\sphere{7}$, where $\liftDtwo$ and $\liftDfour$ are orthogonal. The $\Jo$-invariance also follows by lifting to $\sphere{7}$.

We now prove that $\Dtwo$ and $\Dfour$ are well defined. The former takes a computation with vertically acting isometries, and the latter will be a simple consequence. Consider the map $q \mapsto F_\theta(q) = e^{i \theta}q$ for real $\theta$, and note that $\dd F_\theta (v) = e^{i \theta}v$. Suppose $q_1, q_2 \in \sphere{7}$ are such that $\hpi(q_1) = \hpi(q_2) \in \CP^3$, so that there is a $\theta \in \reals$ such that $q_2 = F_\theta(q_1)$. Since $\hpi \circ F_\theta = \hpi$, the chain rule gives $\dd \hpi_{q_2} (\dd F_\theta)_{q_1} = \dd \hpi_{q_1}$. Similarly, as $q_1 = F_{-\theta}(q_2)$, we find $\dd \hpi_{q_2}  = \dd \hpi_{q_1} \dd F_{-\theta}$. It follows that
\begin{align*}
\dd \hpi_{q_2}( \liftDtwo(q_2) ) = \set{ \dd \hpi_{q_2}( \alpha\, j q_2 + \beta\, k q_2) \mid \alpha, \beta \in \reals } = \set{ \dd \hpi_{q_1}( \alpha\,  e^{-i \theta} j q_2 + \beta\,  e^{-i \theta} k q_2) \mid \alpha, \beta \in \reals }.
\end{align*}
Using that $q_2 = e^{i \theta} q_1$, and separating terms, we obtain after a renaming
\begin{align*}
\dd \hpi_{q_2}( \liftDtwo(q_2) ) &= \set{ \dd \hpi_{q_1}\l[ \l(\alpha\, \cos(2\theta) + \beta\, \sin(2\theta) \r)  j q_1 + \l( \beta\, \cos(2\theta)  - \alpha\, \sin(2\theta)  \r) k q_1 \r] \mid \alpha, \beta \in \reals } \\ 
&= \set{ \dd \hpi_{q_1}\l(  \gamma\, j q_1 + \delta\,  k q_1 \r) \mid \gamma, \delta \in \reals },
\end{align*} 
which is precisely $\dd \hpi_{q_1}( \liftDtwo(q_1) )$. Hence, $\Dtwo$ is well defined. In addition, $\Dfour(p)$ is well-defined, as it is the orthogonal complement of $\Dtwo(p)$ in $T_p(\CP^3)$. 
\end{proof}

\begin{definition}\label{def:nearly_Kahler_\J}
Let $\J:T\CP^3 \to T\CP^3$ be defined by 
\begin{equation*}
\J = 
\begin{cases}
- \Jo & \text{on } \Dtwo, \\
\Jo & \text{on } \Dfour 
\end{cases}
\end{equation*} 
and linearity, where $\Jo$ is the Kähler structure on $\CP^3$.
\end{definition}

The structure $J$ is an almost complex structure compatible with $\FSm$, and it corresponds to the one of the standard twistor construction of \cite{eells1985}. 
Yet, $\FSnabla \J$ is not skew-symmetric, so that $(\CP^3, \FSm, \J)$ is not nearly Kähler. Up to rescaling, the following conditions define the unique family of non-degenerate Riemannian metrics that keeps the two distributions orthogonal:
\begin{equation}\label{eq:definition_family_NK_metric}
g_\metricparam = 
\begin{cases}
\FSm		& \text{on } \Dtwo\times \Dtwo, \\
\metricparam \cdot \FSm	& \text{on } \Dfour\times \Dfour, \\
0	        & \text{on } \Dtwo \times \Dfour
\end{cases}
\end{equation} 
for some real constant $a>0$. Moreover, it can be shown that this family of metrics contains exactly the metrics (up to global rescaling) making $\CP^3$ homogeneous \cite{onishchik_transitive_1963,volper1999}.

Define the tensor field $P$ as
\begin{equation}\label{eq:def_K}
P = - \J \Jo = - \Jo \J = 
\begin{cases}
- \identity & \text{on } \Dtwo, \\
\identity	& \text{on } \Dfour.
\end{cases}
\end{equation} 

\begin{lemma}\label{thm:properties_K}
The tensor field $P$ defined by \eqref{eq:def_K} satisfies $P\Jo = \Jo P = \J$, $P\J = \J P = \Jo$ and $P^2 = \identity$ and is compatible with $g_\metricparam$ for all $\metricparam$. In particular, it is an almost product structure on $\CP^3$ that commutes with $\J$ and $\Jo$.
\end{lemma} 

Note that the distributions $\Dtwo$ and $\Dfour$ are the eigenspaces of $P$ corresponding to eigenvalues $-1$ and $1$, respectively. Therefore, $(\identity - P)/2$ and $(\identity + P)/2$ are the projections to these eigenspaces, and we can re-express $g_\metricparam$ in terms of $\FSm$ and vice versa:
\begin{align}
g_\metricparam(X,Y) &= \frac{1+\metricparam}{2} \FSm(X,Y) + \frac{\metricparam-1}{2} \FSm(PX,Y),  \label{eq:ga_in_terms_of_g0} \\
\FSm(X,Y) &= \frac{1+\metricparam}{2\metricparam} g_\metricparam(X,Y) + \frac{1-\metricparam}{2\metricparam} g_\metricparam(PX,Y). \label{eq:g0_in_terms_of_ga}
\end{align}

\begin{remark}
From \Cref{eq:ga_in_terms_of_g0} together with $P \Jo = \Jo P$ and $g_\metricparam$-compatibility of $P$, we find that $(\CP^3, g_\metricparam, \Jo)$ is a Hermitian manifold for all $\metricparam$. So, all of the spaces $(\CP^3, g_\metricparam, \J)$ and $(\CP^3, g_\metricparam, \Jo)$ are Hermitian manifolds. 
\end{remark}

\subsection{A contact frame on \texorpdfstring{$\CP^3$}{ℂP³}}\label{sec:frame}

In this section, we construct a useful $\FSm$-orthonormal frame, using the complex almost contact structure of $\CP^3$ of \Cref{sec:kahlerCP3}. Note that $U$ can be chosen to be any section of $\Dtwo$ satisfying $\FSm(U,U) = 1$, and hence $g_{a}(U,U) = 1$. Next, choose a section $\chi$ of $\Dfour$, and define the following frame:
\begin{equation}\label{eq:CP3_contact_frame_def}
\begin{aligned}
E_1 &= V, &
E_2 &= U, &
E_3 &= \chi,  &
E_4 &= \Jo \chi,  &
E_5 &= \Phi \chi,  &
E_6 &= \Psi \chi.
\end{aligned}
\end{equation} 
In \cite{blair2000}, $(\CP^3, \FSm, \Jo, \Psi, \Phi, u,v)$ is said to be a complex almost contact metric manifold, and for this reason, we will call the frame \eqref{eq:CP3_contact_frame_def} a contact frame on $\CP^3$. It follows from a straightforward computation that this frame is orthonormal with respect to the Fubini-Study metric $\FSm$. From this, one can observe that $(\identity, \Jo, \Phi, \Psi)$ behave as the quaternions on $\Dfour$. 

\begin{lemma}\label{thm:Levi_civita_connection_kahler_CP3}
Let $\FSnabla$ denote the Levi-Civita connection of the Fubini-Study metric, and consider the frame \eqref{eq:CP3_contact_frame_def}. Let $\zeta_{i} = \FSnabla_{E_i} E_3$, then
\begin{align*}
\FSnabla_{E_i} E_1 &= - \Phi E_i - \sigma(E_i) E_2, &
\FSnabla_{E_i} E_3 &=  \zeta_{i}, &
\FSnabla_{E_i} E_5 &= \delta_{3i}E_1 + \delta_{4i} E_2 -\sigma(E_i)E_6 + \Phi \zeta_{i}, \\
\FSnabla_{E_i} E_2 &= - \Psi E_i + \sigma(E_i) E_1, &
\FSnabla_{E_i} E_4 &= \Jo \zeta_{i}, &
\FSnabla_{E_i} E_6 &= \delta_{3i}E_2 - \delta_{4i} E_1 +\sigma(E_i)E_5 + \Psi \zeta_{i},
\end{align*}
where $\delta_{ij}$ is the Kronecker delta. Moreover, the vector fields $\zeta_{i}$ satisfy 
\begin{equation}\label{eq:CP3_contact_conditions_nabla0}
\FSm(\zeta_{i}, E_1) = -\delta_{i5}, \qquad
\FSm(\zeta_{i}, E_2) = -\delta_{i6}, \qquad
\FSm(\zeta_{i}, E_3) = 0.
\end{equation}
\end{lemma}

\begin{proof}
Observe that $u(\chi) = \FSm(U, \chi) = 0 =v(\chi)$. From \Cref{eq:nabla0_UV,eq:nabla0_psi,eq:nabla0_phi} we then find the expressions for the derivatives of the frame vector fields. Further, as \eqref{eq:CP3_contact_frame_def} is an orthonormal frame with respect to $\FSm$, we have that $\FSm(\FSnabla_{E_k} E_i,E_j) = - \FSm(E_i,\FSnabla_{E_k} E_j)$ for all $i$, $j$ and $k$. From this, we immediately find \eqref{eq:CP3_contact_conditions_nabla0}. 
\end{proof}

\begin{proposition}\label{thm:skew_G_independent_metricparam}
Let $\FSnabla$ and $\nabla^a$ be the Levi-Civita connections of $(\CP^3,\FSm)$ and $(\CP^3,g_\metricparam)$, respectively. The skew-symmetric part of $G_\metricparam = \nabla^a J$ is independent of $a>0$ and is denoted by $G$. The symmetric part $G_a^+$ of $G_\metricparam$ and the difference tensor $D_\metricparam(X,Y) = \nabla^\metricparam_X Y - \FSnabla_X Y$ only take values in $\Dfour$ and are related to $G$ as follows
\begin{align*}
D_\metricparam(X,Y) &= \frac{\metricparam-1}{\metricparam} \frac{\identity+P}{2} G(\Jo X, Y), &
G_\metricparam^+(X,Y) &= \frac{2-\metricparam}{\metricparam} \frac{\identity+P}{2} G(\Jo X, Y).
\end{align*}
\end{proposition}
\begin{proof}
    We first show that the symmetric part of $G_a$ is independent of $a$. 
    With a frame computation, we can show that $D_\metricparam(X,\J Y) + \J D_\metricparam(X,Y)=0$ for all $a>0$. The definition of $D_a$ then gives
    \begin{align*}
        (\nabla^\metricparam \J)(X,Y) = (\FSnabla \J)(X,Y) + D_\metricparam(X,\J Y) - \J D_\metricparam(X,Y) = (\FSnabla \J)(X,Y) - 2 \J D_\metricparam(X,Y).
    \end{align*}
    Symmetry of $D_\metricparam$ implies that $2 G(X,Y) = (\FSnabla \J)(X,Y)-(\FSnabla \J)(Y,X)$ is independent of $a$. 

    For the second part of the proof, it suffices to verify the equalities for vector fields of the frame \eqref{eq:CP3_contact_frame_def}, as both sides of both equalities are tensorial. To do this, compute the Lie brackets of the frame vector fields with \Cref{thm:Levi_civita_connection_kahler_CP3} and use the Koszul formula together with the definition of $D_a$ to compute the components of $D_\metricparam$. The components of $G_a^+$ can then also be calculated by using its definition. That both tensors only take values in $\Dfour$ is a simple observation using the fact that $\frac{\identity+P}{2}$ is the orthogonal projection onto $\Dfour$. 
\end{proof}

The tensor $G$ determines $D_\metricparam$ and $G_\metricparam^+$ via \Cref{thm:skew_G_independent_metricparam}. We now list how $G$ interacts with $\J$, $\Jo$ and $P$, for which we need a lemma. 

\begin{lemma}\label{thm:G_on_distributions}
Let $\Dtwo$ and $\Dfour$ be the orthogonal distributions on $\CP^3$ introduced in \Cref{thm:decomposition_tangent_space_CP3}. If $A$ and $B$ are sections of  $\Dtwo$ and $X$ and $Y$ are sections of $\Dfour$, then 
\begin{align*}
G(A,B) &= 0, &
G(X,Y) &\in \Dtwo, &
G(A,X),\, G(X,A) &\in \Dfour.
\end{align*}
\end{lemma}
\begin{proof}
This is a simple observation when both sides are expressed in the frame \eqref{eq:CP3_contact_frame_def}.
\end{proof}

\begin{proposition}\label{thm:K_J0_through_G}
The tensors $G$, $P$ and $\Jo$ on $\CP^3$ relate to each other as follows:
\begin{align*}
&P G(X,Y) = - G(PX,PY), & &G(PX,Y) = - PG(X,PY), \\
&G(\Jo X,Y) = -PG(X,\Jo Y), & &G(\Jo X,\Jo Y) = PG(X, Y).
\end{align*}
\end{proposition}

\begin{proof}
To show the first equality, let $X$ and $Y$ be any vector fields on $\CP^3$. By \Cref{thm:G_on_distributions}, we know how $G$ behaves on the distributions. Using this fact and the observation that $P = \mathcal{P}_2-\mathcal{P}_1$, where $\mathcal{P}_2=\frac{\identity+P}{2}$ and $\mathcal{P}_1=\frac{\identity-P}{2}$, we see
\begin{align*} 
G(PX, P Y)
&= -P\Big( G(\mathcal{P}_2X,\mathcal{P}_2Y) + G(\mathcal{P}_1X,\mathcal{P}_1Y) +G(\mathcal{P}_2X,\mathcal{P}_1Y) +G(\mathcal{P}_1X,\mathcal{P}_2Y) \Big) \\
&= -P G(X,Y).
\end{align*}
The second equality follows from the fact that $P^2=\identity$ and the third one follows from the first one by using $\Jo=P\J$ and $G(\J X, Y)=G(X, \J Y)$. Finally, the fourth equality follows from the third one by the fact that $\Jo^2=-\identity$. 
\end{proof}

We finish this section by giving the covariant derivative of $\Jo$. 

\begin{proposition} \label{thm:link_nablaJ_nablaJ0}
Let $\nabla^\metricparam$ be the Levi-Civita connection of $(\CP^3,g_\metricparam)$. Then, for all tangent vectors $X$ and $Y$ of $\CP^3$, we have 
\begin{equation*}
(\nabla^\metricparam \Jo)(X,Y) = 2 \frac{a-1}{a} G\left(\frac{\identity-P}{2} X, \frac{\identity+P}{2} Y\right).
\end{equation*}
\end{proposition}

\begin{proof}
Using $\FSnabla \Jo = 0$,the definition of $D_a$ yields $(\nabla^\metricparam \Jo)(X,Y) = D_\metricparam(X, \Jo Y)- \Jo D_\metricparam(X,Y)$. Writing this in terms of $G$ via \Cref{thm:skew_G_independent_metricparam} and using that $\Jo = P \J$, we find 
$$(\nabla^\metricparam \Jo)(X,Y) = \frac{a-1}{a} \frac{\identity+P}{2} G(X-P X, Y).$$ 
Using \Cref{thm:G_on_distributions}, we obtain the result. 
\end{proof}

\subsection{The homogeneous nearly Kähler metric on \texorpdfstring{$\CP^3$}{ℂP³}}

\begin{theorem}\label{thm:CP3_is_NK}
The space $(\CP^3, g_\metricparam, \J)$ is quasi-Kähler for all $a>0$ and it is nearly Kähler if and only if $\metricparam=2$. These spaces are never Kähler. 
\end{theorem}

\begin{proof}
Observe that $G^\metricparam(X,Y) + G^\metricparam(\J X, \J Y)= 0$ for all $a$, so that $(\CP^3, g_\metricparam, \J)$ is always quasi-Kähler. Since the symmetric part of $G_\metricparam$ vanishes if and only if $\metricparam=2$, we find that by the homogeneity of $\CP^3$, the space $(\CP^3, g_2, \J)$ is a strict nearly Kähler manifold.
\end{proof}

\begin{remark}
    The nearly Kähler metric on $\CP^3$ with the conventions of this work is of constant type $\alpha=1$.
\end{remark}

\section{Curvature of \texorpdfstring{$(\CP^3, g_\metricparam)$}{(ℂP³,ga)}} \label{sec:curvature}

In this section, we express the Riemann curvature tensor of $(\CP^3, g_\metricparam)$ in terms of $\Jo$, $\J$ and $P$. In \cite{deschamps2021}, an expression for the curvature for the nearly Kähler case $a=2$ is given (which is equivalent to our expression for that specific case), in terms of the twistor fibration and information on $\sphere{4}$. As such, that expression does not merely depend on intrinsic information about $\CP^3$. With the new description of the homogeneous nearly Kähler $\CP^3$ in terms of the distributions of \Cref{thm:decomposition_tangent_space_CP3}, we will give a new expression for the curvature of the homogeneous nearly Kähler $\CP^3$ and immediately generalize it from $a=2$ to all $a>0$. 

\begin{theorem}\label{thm:curvature_CP3a}
Let $\FSm$ and $\Jo$ be the Fubini-Study metric and the standard Kähler structure of $\CP^3$ and denote by $J$ and $P$ the structures defined in Section \ref{sec:description}. Using the notation $(X \wedge_\metricparam Y)Z = g_\metricparam(Y,Z) X - g_\metricparam(X,Z)Y$, the Riemann curvature tensor of $(\CP^3, g_\metricparam)$ is given by 
\begin{align*}
R^\metricparam(X,Y)Z 
&= \frac{(a-1)(a+2)}{a^2} (X \wedge_\metricparam Y)Z \\
&\quad + \frac{1}{a}\bigg( (X \wedge_\metricparam Y)Z + (\Jo X \wedge_\metricparam \Jo Y)Z + 2g_\metricparam(X,\Jo  Y) \Jo  Z \bigg) \\
&\quad + \frac{1-a}{a^2}\bigg( (X \wedge_\metricparam Y)Z  + (\J X \wedge_\metricparam \J Y)Z + 2 g_\metricparam(X,\J Y)\J Z \bigg) \\
&\quad + \frac{1-a}{a}\bigg( (X \wedge_\metricparam Y)P Z + P (X \wedge_\metricparam Y)Z  - \frac{a+2}{a} P( X \wedge_\metricparam Y) P Z \bigg).
\end{align*}
\end{theorem}

\begin{proof}
It suffices to verify the equality for a local frame. To do this, note that the curvature of $(\CP^3, g_\metricparam)$ and of $(\CP^3,\FSm)$ are linked by
\begin{multline}\label{eq:riemann_tensors_difference_tensor}
R_\metricparam(X,Y)Z 	= \FS{R}(X,Y)Z + ((\nabla^\metricparam D_\metricparam)(X,Y,Z)-(\nabla^\metricparam D_\metricparam)(Y,X,Z)) \\ - (D_\metricparam(X, D_\metricparam(Y,Z))-D_\metricparam(Y, D_\metricparam(X,Z))),
\end{multline}
where $D_\metricparam$ is the difference tensor. In the contact frame \Cref{eq:CP3_contact_frame_def}, we can check that the right-hand-side of the previous equation matches the expression stated in the theorem. 
\end{proof}

The following is an immediate consequence of \Cref{thm:curvature_CP3a}.

\begin{corollary}\label{thm:ricci_scalar_sectional_curvatures_CP3a}
The Ricci curvature and the non-normalized scalar curvature of $(\CP^3, g_\metricparam)$ are given by
\begin{align*}
\Ric_\metricparam(X,Y) &= 4 \l( 1+ \frac{1}{a^2} \r)g_\metricparam\l(\frac{\identity-P}{2} X, Y\r) + 4 \frac{3a-1}{a^2} g_\metricparam\l(\frac{\identity+P}{2} X, Y\r), & S_\metricparam &= 8 \frac{a^2 +6 a-1}{a^2}.
\end{align*}
\end{corollary}

In particular, only for $a=1$ and $a=2$ (corresponding to the Kähler and the nearly Kähler case for the appropriate almost complex structures) $(\CP^3, g_\metricparam)$ is Einstein. This is in agreement with a classical result of \cite{ziller1982}. Moreover, for orthogonal unit vector fields $X$ and $Y$, the sectional curvature of $(\CP^3, g_\metricparam)$ is given by
\begin{equation}\label{eq:sec_CP3a}
\begin{aligned}
    &\sec_\metricparam(X,Y) 
        = \frac{ a^2 +a-1 }{a^2} + \frac{3}{a}  g_\metricparam(X, \Jo Y)^2 + 3 \frac{1-a}{a^2}  g_\metricparam(X, \J Y)^2 \\
	    &\quad
	    + \frac{1-a}{a}\bigg( g_\metricparam(Y,PY) + g_\metricparam( P X, X)  - \frac{a+2}{a} \l( g_\metricparam(Y,PY) g_\metricparam( PX,X) - g_\metricparam(X,PY)^2 \r) \bigg).
\end{aligned}
\end{equation}

\section{Isometries of \texorpdfstring{$(\CP^3, g_\metricparam)$}{(ℂP³,ga)}}\label{sec:isometries}

In this section, we characterise the isometries of $(\CP^3, g_\metricparam)$. Here, an isometry is defined to be a smooth metric preserving map. In particular, we do not require that it preserves any almost complex or almost product structure. Note that for $a=1$, i.e., the Kähler case when equipped with $\Jo$, the full isometry group is known. To describe it, let $\epsilon \in \integers_2 = \integers/2 \integers$. For $p \in \mathbb C^4$, define 
$\epsilon(p) = p$ when $\epsilon=0$, and $\epsilon(p) = \overline{p}$ when $\epsilon=1$. Then $\SU(4) \rtimes \integers_2$ acts on $\complexs^4$ by $(A, \epsilon)\cdot p = A\, \epsilon(p)$. Note that $(A, \epsilon) \circ (B, \nu) = ( A \epsilon(B) , \epsilon + \nu )$. We then also have an induced action on $\CP^3$. The full isometry group of $(\CP^3, g_1)$ is given by
\[ \Iso(\CP^3, \FSm) = \set{ \phi \mid \hpi \circ A = \phi \circ \hpi \text{ and } A \in \SU(4)  } \rtimes \integers_2. \]

We will prove that for $a \neq 1$, every isometry of $(\CP^3, g_\metricparam)$ preserves the almost product structure $P = -\J \Jo$. 

\begin{lemma}\label{thm:partial_conservation_of_K} 
Let $\phi: (\CP^3,g_\metricparam) \to (\CP^3,g_\metricparam)$ be an isometry and let $P$ be the almost product structure defined in \eqref{eq:def_K}. For every $X \in \mathfrak{X}(\CP^3)$ with $g_\metricparam(X,X)=1$, write $\alpha = g_\metricparam(X,PX)$ and $\beta = g_\metricparam(X,P\varphi_{\ast}X)$. Then $(a-1)(\alpha - \beta) \l( -\alpha - \beta + a (-2 + \alpha + \beta) \r) = 0$.
\end{lemma}
	
\begin{proof}
Let $X$ be an arbitrary vector field of unit length. Then
\[ \sec_\metricparam(X,\Jo X) = \frac{ a^2 +a-1 }{a^2} + \frac{3}{a} + 3 \frac{1-a}{a^2}  \alpha^2 + \frac{1-a}{a}\bigg( 2\alpha  - \frac{a+2}{a} \alpha^2 \bigg) \]
by \eqref{eq:sec_CP3a}. The result follows from the fact that isometries preserve sectional curvature.
\end{proof}

\begin{theorem}\label{thm:conservation_of_K}
For all $a\neq 1$, all isometries of $(\CP^3, g_\metricparam)$ preserve the almost product structure $P$, defined by \eqref{eq:def_K}. 
\end{theorem}

\begin{proof}
Write $\tilde{P} = (\phi^{-1})^*P$ for a fixed isometry $\phi$. Note that $\tilde{P}$ is an almost product structure because $\tilde{P}^2 = \identity$. Since eigenvectors of $P$ are mapped to eigenvectors of $\tilde{P}$ with identical eigenvalues, we have two orthogonal eigendistributions $\tilde{\mathcal{D}}_1^2$ and $\tilde{\mathcal{D}}_2^4$. Take $X \in \tilde{\mathcal{D}}_1^2$ and $Y \in \tilde{\mathcal{D}}_2^4$, both of unit $g_a$-length, and decompose them as $X = A + \chi$ and $Y = B + \upsilon$ for $A,B \in \mathcal{D}_1^2$ and $\chi,\upsilon \in \mathcal{D}_2^4$. Note that $\tilde{P} X = - X$, $P X = -A + \chi$, $\tilde{P} Y = Y$ and $P Y = -B + \upsilon$. Furthermore, since $X,Y$ are of unit $g_\metricparam$-length, we have $g_a(A,A)$, $g_a(\chi,\chi)$, $g_a(B,B)$, $g_a(\upsilon,\upsilon) \leq 1$.

We now consider two cases.

\textit{Case 1: $2 a \geq 1$.} 
By \Cref{thm:partial_conservation_of_K}, we know that either $g_\metricparam(PX,X) = g_\metricparam(\tilde{P}X,X)$ or that $g_\metricparam(PX,X) = \frac{2 a}{a-1} - g_\metricparam(\tilde{P}X,X)$. The former implies $g_\metricparam(\chi,\chi) + 1 = g_\metricparam(A,A)$, so that $\chi = 0$, while the latter implies $g_\metricparam(A,A) = \frac{a}{1-a}$. If $2 a >1$, this last equation can never be true as $0 \leq g_\metricparam(A,A) \leq 1$. For $2a=1$, we find $g_\metricparam(A,A)=1$, so that again $\chi=0$. It follows that $\tilde{\mathcal{D}_1^2} \ni X =A \in \mathcal{D}_1^2$. Since $X$ was arbitrary, we get $\tilde{\mathcal{D}}_1^2 = \mathcal{D}_1^2$. Since $\mathcal{D}_2^4 = (\mathcal{D}_1^2)^\perp$, we also get that $\mathcal{D}_2^4$ is preserved under each isometry. Hence, $\tilde{P} = P$ in this case.

\textit{Case 2: $2 a < 1$.} 
By \Cref{thm:partial_conservation_of_K}, we know that either $g_\metricparam(PY,Y) = g_\metricparam(\tilde{P}Y,Y)$ or $g_\metricparam(PY,Y) = \frac{2 a}{a-1} - g_\metricparam(\tilde{P}Y,Y)$ has to hold. Similarly as before, the former implies $B=0$. The latter implies $g_\metricparam(\upsilon,\upsilon) = \frac{a}{a-1}$. If $2 a <1$, this last equation can never be true as $0 \leq g_\metricparam(\upsilon,\upsilon) \leq 1$. It follows that $\tilde{\mathcal{D}_2^4} \ni Y = \upsilon \in \mathcal{D}_2^4$. Since $Y$ was arbitrary, we find that also in this case $\tilde{P} = P$.
\end{proof}

\begin{corollary}\label{thm:isometries_NK_characterised}
For all $a \neq 1$, the isometries of $(\CP^3, g_\metricparam)$ are those isometries of the (Kähler) manifold $(\CP^3,g_1)$ that preserve the almost product structure $P$ defined by \eqref{eq:def_K}.
\end{corollary}

\begin{proof}
By \Cref{eq:g0_in_terms_of_ga} and \Cref{thm:conservation_of_K}, every isometry of $(\CP^3, g_\metricparam)$ is an isometry of $(\CP^3,g_1)$. By \Cref{thm:conservation_of_K} this isometry preserves the almost product structure. 

Conversely, if the almost product structure is preserved by an isometry of $(\CP^3,g_1)$, then \Cref{eq:g0_in_terms_of_ga} implies that this map is also an isometry of $(\CP^3,g_\metricparam)$.
\end{proof}

\begin{corollary}\label{thm:J_J0_under_isometries}
An isometry of $(\CP^3, g_\metricparam)$ maps $\J$ to $\epsilon \J$ with $\epsilon \in \set{\pm 1}$. Moreover, this isometry maps $\Jo$ to $\epsilon \Jo$. 
\end{corollary}

\begin{proof}
Note that every isometry of nearly Kähler $\CP^3$ maps $\J$ to $\epsilon \J$, which is an immediate corollary of Butruille's classification. Since the isometry groups of $(\CP^3, g_\metricparam)$ are the same for all $a \neq 1$, we find that $\J$ maps to $\epsilon \J$ for all isometries of $(\CP^3, g_\metricparam)$. Since $\Jo = \J P$, we have for all isometries $\phi$ that $(\phi^{-1})^*\Jo = ((\phi^{-1})^*\J) \circ ((\phi^{-1})^*P)$. Since $P$ is preserved, we have $(\phi^{-1})^*\Jo = \epsilon \Jo$.
\end{proof}

The following result can be verified by a direct computation and agrees with a result from \cite{shankar2001a} for the nearly Kähler $\CP^3$.  
\begin{theorem}
The isometry group of $(\CP^3, g_\metricparam)$, with $a \neq 1$, is
\[ \Iso(\CP^3, g_\metricparam) = \set{ \phi \mid \hpi \circ A = \phi \circ \hpi \text{ and } A \in \Sp(2)\cong \SU(4) \cap \Sp(4, \complexs)  } \rtimes \integers_2. \] 
\end{theorem}

\begin{remark}
From the above explicit description of $\Iso(\CP^3, g_\metricparam)$, one can compute that the constant $\varepsilon$ introduced in \Cref{thm:J_J0_under_isometries} satisfies $\varepsilon=1$ on the identity component of the group and $\varepsilon=-1$ on the other component.
\end{remark}

\section{Lagrangian submanifolds of the nearly Kähler \texorpdfstring{$\CP^3$}{ℂP³}}
\label{sec:lagrangian_submanifolds}

In this section, we study Lagrangian submanifolds of the nearly Kähler $\CP^3$. Pioneering work on this topic can be found in \cite{aslan2023}, containing results from the author's PhD thesis \cite{aslan2022}. Most of the results that we present hereafter were obtained independently and simultaneously in the first author's MSc thesis \cite{liefsoens2024thesis}. In \cite{aslan2023}, an angle function is introduced, and through abstract means, the existence of most of the examples given explicitly in this work is established. Moreover, those Lagrangian submanifolds allowing a non-trivial $\SU(2)$-action are classified. We extend these results by giving a geometric interpretation for the angle function, giving explicit immersions and proving that all examples are (locally) rigid. Moreover, we consider all possible subgroups of the isometry group of the nearly Kähler $\CP^3$ in order to classify the extrinsically homogeneous Lagrangian submanifolds. 

\begin{remark}
For the remainder of the paper, we will write $g=g_2$ for the nearly Kähler metric on $\CP^3$, and $\nabla = \nabla^2$ for the corresponding Levi-Civita connection. 
\end{remark}

\subsection{General properties of Lagrangian submanifolds of nearly Kähler six-manifolds}

Lagrangian submanifolds generally have many interesting properties. For example, for any Lagrangian submanifold with second fundamental form $h$ of a Lagrangian submanifold of any nearly Kähler manifold $(M,g,\J)$, the tensor field $(X,Y,Z) \mapsto g(h(X,Y), \J Z)$ is totally symmetric. Furthermore, $G= \nabla \J$ maps two tangent vectors to a normal vector. Also, we have the following.

\begin{proposition}\label{thm:R_perp_lagrangian}
Let $\mathcal{L}$ be a Lagrangian submanifold of a nearly Kähler manifold $(M,g,\J)$ and let $G = \nabla \J$. Let $R^\mathcal{L}$ the induced curvature tensor of $\mathcal{L}$. Then, the curvature of the normal bundle $R^\perp$ is given by
\begin{equation}\label{eq:R_perp_lagrangian}
R^\perp(X,Y) \J Z = (\nabla^\perp G)(X,Y,Z) - (\nabla^\perp G)(Y,X,Z) + \J R^\mathcal{L}(X,Y)Z.
\end{equation}
\end{proposition}

For six-dimensional nearly Kähler manifolds, the properties of Lagrangian submanifolds are even stronger: every Lagrangian is then orientable and minimal, parallel Lagrangians are totally geodesic, and $G$ behaves like a cross-product on $\reals^3$ as the following result shows.  

\begin{proposition} \label{thm:JG_fully_anti_symmetric_lagsub}
Let $\mathcal{L}$ be a Lagrangian submanifold of a nearly Kähler six-manifold $(M,g,\J)$ of constant type $\alpha$ and let $G=\nabla \J$. For every orthonormal frame $(e_1,e_2,e_3)$ on $\mathcal{L}$, possibly after changing $e_3$ to $-e_3$, it holds that $\J G(e_i, e_j) = \alpha \sum_k \epsilon_{ijk} e_k$, where $\epsilon_{ijk}$ is the Levi-Civita symbol.
\end{proposition}

Note that this result indeed holds for all six-dimensional nearly Kähler manifolds, as they are all of constant type. 

\subsection{Definition of the angle function}\label{sec:angle}

From now on, we consider Lagrangian submanifolds of the nearly Kähler $\CP^3$. 

\begin{proposition}\label{prop:AB}
Let $\mathcal{L}$ be a Lagrangian submanifold of the nearly Kähler $\CP^3$, and let $P$ be the almost product structure defined in \eqref{eq:def_K}. Define $(1,1)$-tensor fields $A$ and $B$ on $\mathcal L$ by the orthogonal decomposition $PX = AX+\J BX$ for any $X$ tangent to $\mathcal{L}$. Then $A$ is symmetric, $B$ is skew-symmetric, $A$ and $B$ anti-commute, i.e. $AB+BA =0$, and they satisfy $A^2-B^2 = \identity$.
\end{proposition}

\begin{proof}
Let $X,Y \in T\mathcal{L}$. From metric compatibility of $\J$ and $\Jo$, we find that $A$ is symmetric. Further, by composing $P$ by $\J$, we find $B X = \J(A-P)X$. Again from metric compatibility, we then find $B$ is skew-symmetric. 

The other two properties can be shown simultaneously. Since $P$ is an almost product structure, 
\begin{align*}
X = P^2 X &= P ( AX + \J B X ) = (A^2 -B^2) X+ \J (AB + BA) X.
\end{align*} 
Comparing tangent and normal parts finishes the proof.
\end{proof}

The following theorem introduces a local angle function on a Lagrangian submanifold of the nearly Kähler $\CP^3$.

\begin{theorem}\label{thm:intro_lagframe_ABform}
Let $\mathcal L$ be a Lagrangian submanifold of the nearly Kähler $(\CP^3,g, \J)$ and define $A$ and $B$ as in Proposition \ref{prop:AB}. Then there exists a local orthonormal frame $( e_1,  e_2,  e_3 )$ of $\mathcal{L}$ such that $A$ and $B$ take the following form:
\begin{align*}
A  e_1 &=  e_1, & A  e_2 &= \cos(2 \theta)  e_2, & A  e_3 &= -\cos(2 \theta)  e_3, \\
B  e_1 &= 0, & B  e_2 &= -\sin(2 \theta)  e_3, & B  e_3 &= \sin(2 \theta)  e_2,
\end{align*}
where $\theta$ is a local smooth function. At any single given point, this angle can be taken in the range $\theta \in \l[0, \frac{\pi}{4}\r]$. 

This frame also satisfies $G( e_i,  e_j) = \sum_k \epsilon_{ijk} \J  e_k$, with $\epsilon_{ijk}$ the Levi-Civita symbol. In addition, we can write $ e_2 = \cos\theta \, W + \sin \theta \, U$ and $ e_3 = \sin\theta \, \J W - \cos \theta \, \J U$, with $ e_1,W \in \Dfour$, $W = \Phi  e_1$ and $U \in \Dtwo$. The vector fields $U$ and $W$ are of $g$-unit length and $\Phi$ is as in \Cref{eq:complex_almost_contact_structure}.
\end{theorem}
\begin{proof}
Since $\dim_\reals T_p\mathcal{L} = 3$, $\dim_\reals \Dfour(p) = 4$ and $\dim_\reals T_p \CP^3 = 6$, for any point $p$ there is a non-zero vector in the intersection of $\Dfour(p)$ and $T_p\mathcal{L}$. This gives rise to a non-zero vector field in $\Dfour \cap T\mathcal{L}$, which we can suppose to be of unit length.
Recall that $\J = \Jo$ on $\Dfour$, from which the expressions for $A  e_1$ and $B  e_1$ follow. Since $A$ is symmetric, we can find vector fields $ e_2,  e_3$ so that $A$ is diagonal in the frame $( e_1, e_2,  e_3)$, with $A  e_i = a_i  e_i$ and $a_1 = 1$. Due to skew-symmetry of $B$, we then find that in this frame $B  e_2 = -b  e_3$ and $B  e_3 = b  e_2$. Using that $A^2-B^2 = \identity$, we get the following conditions: $a_1^2 + b^2 = 1$, $a_2^2 + b^2 = 1$ and $(a_1 + a_2) b = 0$. For dimensional reasons, it follows that $a_1=-a_2$. We write $a_1 = -a_2 = a$ and have $a^2 + b^2 = 1$. We now restrict to a single given point $p\in \mathcal{L}$. We can assume that $a \geq 0$ at $p$, by interchanging $ e_2$ and $ e_3$. Moreover, by changing $ e_2$ to $- e_2$, we can assume simultaneously that $b \geq 0$ at $p$.
Since $a^2 + b^2 = 1$ everywhere, there is a local function $\theta$ such that $a = \cos(2 \theta)$ and $b=\sin(2 \theta)$. Since we can rearrange at $p$ so that $a,b \geq 0$, we have that $\theta(p) \in \l[0,\frac{\pi}{4}\r]$. 

Note that \Cref{thm:JG_fully_anti_symmetric_lagsub} implies that we can change the sign of $ e_1$ (which does not change $A$ or $B$) to obtain $G( e_i,  e_j) = \sum_k \epsilon_{ijk} \J  e_k$ everywhere, where we used that the nearly Kähler $\CP^3$ is of constant type 1. 
    
Denote the projections on $\Dtwo$ and $\Dfour$ by $\mathcal{P}_1$ and $\mathcal{P}_2$, respectively. We have
\begin{align*}
\mathcal{P}_1  e_2 &=  \sin\theta \cdot \l( \sin\theta  e_2 + \cos\theta \J  e_3 \r), & 
\mathcal{P}_2  e_2 &=  \cos\theta \cdot \l( \cos\theta  e_2 - \sin\theta \J  e_3 \r), \\ 
\mathcal{P}_1  e_3 &= - \cos\theta \cdot \J\l( \sin\theta  e_2 + \cos\theta \J  e_3 \r), & 
\mathcal{P}_2  e_3 &=  \sin\theta \cdot \J \l( \cos\theta  e_2 - \sin\theta \J  e_3 \r).
\end{align*}
We see that by defining the unit length vector fields $U$ and $W$ as $U = \cos\theta \, \J  e_3 + \sin\theta \,  e_2  \in \Dtwo$ and $W = - \sin\theta \, \J  e_3 + \cos\theta \,  e_2 \in \Dfour$, we can rewrite $e_2$ and $e_3$ as stated in terms of $U$ and $W$.

Finally, we show that $W = \Phi  e_1$. Recall the contact frame $(E_1,\ldots,E_6)$ of \eqref{eq:CP3_contact_frame_def}. We rescale the frame vectors to be of unit length for the nearly Kähler metric $g$:
\begin{align}\label{eq:renormalised_frame_def}
\hat{E}_i &= E_i, \quad i \leq 2, &
\hat{E}_i &= \frac{1}{\sqrt{2}}E_i, \quad i \geq 3.
\end{align}
Recall the freedom to choose $E_1$ and $E_3$, fixing the other frame vectors. In this case, choose $\hat{E}_3= e_1$ and $\hat{E}_1=\J U$. Then 
\begin{align}\label{eq:renormalised_frame_def_properties}
\hat{E}_2 &= -\J\hat{E}_1, &
\hat{E}_4 &= \J \hat{E}_3, &
\hat{E}_5 &= \Phi \hat{E}_3, &
\hat{E}_6 &= \J \Phi \hat{E}_3.
\end{align}
	
Since $W$ lies in $\Dfour$, it cannot have components in the direction of $\hat{E}_1$ or $\hat{E}_2$. Further, expressing that $ e_i$ must be orthogonal to $ e_j$ ($j\neq i$) and $\J e_k$, one finds that $W = e^{\varphi \J} \Phi  e_1$. Here, $e^{\varphi \J} = \cos \varphi \identity + \sin \varphi \J$. We now show that there is only one $\varphi$ possible, so that $W$ is fixed. 
We calculate in the contact frame and find $\J G( e_1, e_2) = -e^{-\varphi \J}  e_3$.
Since $G( e_1, e_2) = \J  e_3$, per construction of the frame $(e_i)$, we find that $e^{-\varphi \J}=\identity$, and hence $W = \Phi e_1$.
\end{proof}

\begin{remark}\label{remark:interpretation_angle} \Cref{thm:intro_lagframe_ABform} tells us that there is always at least a one-dimensional overlap between $T\mathcal{L}$ and $\Dfour$, given by the vector field $e_1$. Even stronger, in most cases, this is the only overlap. 
In the proof of \Cref{thm:intro_lagframe_ABform}, we show that
\begin{align*}
U &= \cos\theta \ \J  e_3 + \sin\theta \  e_2  \in \Dtwo, & W &= - \sin\theta \ \J  e_3 + \cos\theta \  e_2 \in \Dfour .
\end{align*} 
Hence, the angle function $\theta$ gives the rotation of the plane $\vct\{ e_2, J e_3 \}$ related to the Lagrangian submanifold with respect to the plane $\vct\{ U, W \}$ related to the two canonical distributions on $\CP^3$. Specifically for $\theta=0$, we get a two-dimensional overlap between $T\mathcal{L}$ and $\Dfour$. 
\end{remark}

\begin{remark}
The angle introduced in \Cref{thm:intro_lagframe_ABform} has a clear geometric origin and interpretation, see \Cref{remark:interpretation_angle}. In \cite{aslan2023}, an angle is introduced as well for Lagrangian submanifolds, albeit in a completely different way. A priori, it is not obvious how both angle functions are related. However, given the coming results of submanifolds with specific constant angle functions, we can deduce that the angle in \cite{aslan2023} is equal to $\pi/4$ minus the angle of \Cref{thm:intro_lagframe_ABform}.
\end{remark}

\begin{remark}\label{remark:special_angles}
The angles $\theta = 0$ and $\theta=\pi/4$ are special. If a Lagrangian submanifold of the nearly Kähler $\CP^3$ satisfies $\theta = 0$ everywhere, then, and only then, it is also Lagrangian in the Kähler $\CP^3$. These submanifolds have been characterised in \cite{storm2020}: they correspond (via the twistor fibration) to superminimal surfaces in $\sphere{4}$. The other extreme is that $\theta = \pi/4$ everywhere. In that case, and only in that case, the Lagrangian submanifold is a CR submanifold for the Kähler $\CP^3$. In particular, there is a two-dimensional involutive and complex distribution along the submanifold. In other words, the submanifold is foliated by complex curves.
\end{remark}

\subsection{Examples of Lagrangian submanifolds of the nearly Kähler \texorpdfstring{$\CP^3$}{ℂP³}}

\begin{example}[The real projective space]
    We write $\sphere{3}\to \reals P^3: (x_1, x_2, x_3, x_4) \mapsto [(x_1, x_2, x_3, x_4)]$ for the real Hopf fibration and have the following immersion:
    \begin{equation}\label{eq:RP3}
    \reals P^3 \to \CP^3: [(x_1, x_2, x_3, x_4)] \mapsto \hpi( x_1, x_2, x_3, x_4 ).
    \end{equation}
    In \cite{aslan2023}, it is shown that this is the unique complete totally geodesic Lagrangian submanifold of the nearly Kähler $\CP^3$. It is even the unique totally geodesic submanifold of dimension three or higher \cite{lorenzo-naveiro2024}. The angle function introduced in Theorem \ref{thm:intro_lagframe_ABform} for this example is constant and equal to $\theta = 0$. Under the correspondence of \cite{storm2020}, this Lagrangian corresponds to a totally geodesic two-sphere of $\sphere{4}$.
\end{example}

\begin{example}[The Chiang Lagrangian]\label{sec:Chiang}
The next example was first described in \cite{chen1996}, and later revisited in \cite{chiang2004}. Nowadays, it is known as the Chiang Lagrangian. He showed that it is a quotient of $\reals P^3$ with the dihedral group of order six. In \cite{chen1996}, the immersion was made more concrete: lift the Calabi curve $\CP^1 \to \CP^3:z \mapsto \hpi(1, \sqrt{3} z, \sqrt{3} z^2, z^3)$ to $\sphere{7}$, apply an isometry of the round $\sphere{7}$ to make this lift horizontal for the Hopf fibration (in \cite{chen1996} another projection is used, but this is equivalent), and project this down to $\CP^3$ again to obtain a three-dimensional Lagrangian. 

We give the following explicit immersion of the Chiang Lagrangian:
\begin{multline}\label{eq:Chiang}
\reals \times \sphere{2} \subset \reals\times \reals^3 \to \CP^3: (t, (\alpha_1,\alpha_2,\alpha_3)) \mapsto \\ 
\hpi \l( \mqty( 
\sqrt{3} \sin ^2{t} \l(2 \alpha_1 \alpha_2 \cos {t}+\alpha_3 \l(2 \alpha_2^2+\alpha_3^2-1\r) \sin {t}\r) \\
\sqrt{3} \sin{t} \l(\alpha_3 \sin{t} (2 \alpha_2 \cos{t}-\alpha_1 \alpha_3 \sin{t} ) +\alpha_1 \cos ^2{t}\r) \\
\frac{1}{4} \l(-6 \alpha_3^2 \sin{t} \sin (2 t)+3 \cos{t} + \cos (3 t)\r) \\
\alpha_2 \l(-4 \alpha_2^2-3 \alpha_3^2+3\r) \sin ^3t   ) \r. \\ 
+ \l. i \sin{t} \mqty(
\sqrt{3} \l(\alpha_3 \sin{t} (2 \alpha_1 \cos{t}+\alpha_2 \alpha_3 \sin{t} ) -\alpha_2 \cos ^2{t}\r) \\
\sqrt{3} \sin{t} \l(\l(2 \alpha_2^2+\alpha_3^2-1\r) \cos{t} -2 \alpha_1 \alpha_2 \alpha_3 \sin{t} \r) \\
-\alpha_1 \l(4 \alpha_2^2+\alpha_3^2-1\r) \sin^2{t} \\
\alpha_3^3 \sin ^2{t} -3 \alpha_3 \cos^2{t}  )  \r). 
\end{multline}
By introducing coordinates on $\sphere{2}$, we can easily check that the argument of $\hpi$ in \eqref{eq:Chiang} is horizontal. Moreover, we then also see that $\FSm(X, \Jo Y) = 0$ and $0 = g(X, \Jo Y)$ for all $X,Y$ tangent to the immersed submanifold. Since $g(X, \Jo Y) = g(P X, J  Y) = g(BX, Y)$, we then find that $B = 0$, so that $\theta = 0$. 

Since the Chiang Lagrangian has vanishing angle function, it has to correspond to a superminimal surface in $\sphere{4}$. Recall the Veronese surface 
\[ \reals P^2 \to \sphere{4} \subset \reals^5 : [(x, y, z)] \to \frac{1}{\sqrt{3}} \l( y z, x z, x y, \frac{x^2-y^2}{2}, \frac{x^2+y^2-2z^2}{2 \sqrt{3}} \r), \]
from \cite{dillen2000}, for example. Using Theorem 1 of \cite{kenmotsu1983}, one can recognise the projection of the Chiang Lagrangian as the Veronese surface. 
\end{example}

\begin{example}[A Riemannian product $\sphere{2}\times \sphere{1}$]
    Consider the following parametrisation of a product of $\sphere{2}$ and $\sphere{1}$:
    \begin{equation}\label{eq:S2xS1}
    [0, 2 \pi] \times [0, \pi] \times [0, 2\pi] \to \CP^3 : (t, u,v) \mapsto \hpi( p_0(u,v) + e^{i t} f(u,v) ),
    \end{equation}
    \begin{align*}
    p_0(u,v) &= \frac{1}{6}
    \mqty(
    \l(3+\sqrt{3}\r) \cos{u} \\
    \sqrt{6}\, i e^{i v } \sin{u} \\
    \sqrt{6}\, i \cos{u} \\
    \l(3+\sqrt{3}\r) e^{i v } \sin{u}
    ),
    &
    f(u,v) &= \frac{1}{6}
    \mqty(
    \l(3 - \sqrt{3}\r) \cos{u} \\
     \sqrt{6}\, i e^{i v } \sin{u} \\
    -\sqrt{6}\, i \cos{u} \\
    -\l(3-\sqrt{3}\r) e^{i v } \sin{u}
    ).
    \end{align*}
    Note that for fixed $t$, we get a totally geodesic $\complexs P^1$ in the Kähler $\CP^3$. It follows that the total immersion gives rise to a CR submanifold of the Kähler $\CP^3$ and thus from \Cref{remark:special_angles} that this Lagrangian has constant angle $\theta = \pi/4$. 
\end{example}

\begin{example}[A Berger sphere]
    Consider the immersion
    \begin{equation}\label{eq:berger}
    \sphere{3} \times \reals \subset \complexs^2 \times \reals \to \CP^3 : (z,w, t) \mapsto \hpi\l( \sqrt{\frac{2}{3}} \l( z,  w, \frac{1}{\sqrt{2}} e^{i t}, 0 \r) \r).
    \end{equation}
    Denote the argument of $\hpi$ in \eqref{eq:berger} by $F(z,w,t)$. Consider the vector fields $X_1$, $X_2$ and $X_3$ on $\sphere{3}$ given by $X_1(z,w) = (i z, iw)$, $X_2(z,w) = (- \bar{w}, \bar{z})$ and $X_3(z,w) = (- i \bar{w}, i \bar{z})$. They come from the quaternionic structure on $\sphere{3} \subset \quaternions$. Notice that $\dd F(X_1) - i F + \partial_t F = 0$, so that $\hpi(\Im(F))$ is indeed three-dimensional in $\CP^3$. 
    
    Now let $E_1 = \partial_t F - \frac{1}{3} i F$, $E_2 = \dd F(X_2)$ and $E_3 = \dd F(X_3)$. Note that these vector fields are horizontal and mutually orthogonal (for the round $\sphere{7}$, and $\sphere{7}$ with a lift of $g$). Their lengths are $2 \sqrt{\FSm(E_1,E_1)} = \sqrt{g(E_1,E_1)} = 4/9$ and $4/3 \sqrt{\FSm(E_i,E_i)} = \sqrt{g(E_i,E_i)} = 8/9$ for $i=2,3$. Hence, this is indeed a Berger sphere. Moreover, $E_1 \in \lift{\Dfour}$ and $E_2 = i E_3$. As such, this Lagrangian has $\theta = \pi/4$ everywhere. 
\end{example}

In \cite{aslan2023}, it is shown that the preceding two examples are the unique examples such that our angle function is constant and equal to $\pi/4$. In our approach, this follows immediately from the fundamental equations expressed in the frame, see also the next result, where $h^1_{11} = 1$ corresponds to $\sphere{2}\times\sphere{1}$ and $h^1_{11} = -1/2$ to the Berger sphere. 

\begin{lemma}\label{thm:lagr_pi/4_possibilities}
Let $\mathcal{L}$ be a Lagrangian submanifold of the nearly Kähler $\CP^3$ and let $( e_1, e_2, e_3)$ be a local frame on $\mathcal L$ as in \Cref{thm:intro_lagframe_ABform}, with constant angle function $\theta = \frac{\pi}{4}$. Let $h_{ij}^k = g(\nabla_{ e_i} e_j, \J  e_k)$ and $\omega_{ij}^k = g(\nabla_{ e_i} e_j,  e_k)$ denote the components of the second fundamental form and the connection coefficients, respectively. Then, up to the usual symmetries, the only possibly non-zero components are
\begin{align*}
    \omega_{12}^3, && \omega_{22}^3, && \omega_{33}^2, && 
    \omega_{23}^1 = \omega_{31}^2 &= -\frac{1}{2} + \frac{1}{2} h^1_{11}, &
    h^1_{22} &= - \frac{1}{2} h^1_{11} .
\end{align*}
Moreover, $h^1_{11}$ can only be $1$ or $-1/2$. 
\end{lemma}	

\begin{proof}
The first part is a direct computation. The restrictions on $h^1_{11}$ follow from the Codazzi equation, as $g( \tilde{R}( e_1, e_2) e_1 - (\nabla^\perp_{ e_1} h)( e_2, e_1) + (\nabla^\perp_{ e_2} h)( e_1, e_1),\, \J e_3 ) = \frac{1}{2}( h^1_{11} - 1 )( h^1_{11} + \frac{1}{2} )$ should vanish identically.
\end{proof}


\begin{example}[The exceptional homogeneous Lagrangian]

    We give one final example. Let $f(\beta) = 1 + e^{2 i \beta} \sqrt{15}$ and consider the following matrix in terms of the parameters $\alpha$ and $\beta$:
    \begin{multline*}
    M(\alpha, \beta) =  \\
    \frac{1}{3} \mqty(
    -7 i \cos \alpha & e^{-i \beta } f(\beta ) \sin \alpha & -\frac{(5+f(-\beta )) i e^{i \beta} \sin \alpha }{\sqrt{3}} & 2 \sqrt{5} \cos \alpha \\
    -e^{i \beta } f(-\beta ) \sin \alpha & 7 i \cos \alpha & -2 \sqrt{5} \cos \alpha & \frac{(f(\beta )+5) i e^{-i \beta} \sin \alpha }{\sqrt{3}} \\
    -\frac{i e^{-i \beta } (5+ f(\beta )) \sin \alpha}{\sqrt{3}} & 2 \sqrt{5} \cos \alpha & -i \cos \alpha & e^{i \beta } (f(-\beta )-6) \sin \alpha \\
    -2 \sqrt{5} \cos \alpha & \frac{i e^{i \beta } (f(-\beta )+5)}{\sqrt{3}} \sin \alpha & -e^{-i \beta } (f(\beta )-6) \sin \alpha & i \cos \alpha
    ).
    \end{multline*} 
    Note that $\overline{M(\alpha, \beta)}^T + M(\alpha, \beta) =0$, so that $M(\alpha, \beta)$ is an element of $\sp(2)$.
    Its eigenvalues are $\pm i$ and $\pm 3i$, so that 
    \begin{equation*}
    e^{t M } = \frac{1}{8} \l( 
    (M^2+9) (\sin{t}\, M + \cos{t} \, \identity )- \frac{M^2 + 1}{3} (\sin{(3 t)}\, M +3 \cos{(3 t)} \, \identity  ) 
    \r),
    \end{equation*}
    with $M = M(\alpha, \beta)$. We have the following Lagrangian immersion:
    \begin{equation}\label{eq:EHL}
    [0, 2 \pi] \times [0, 2 \pi] \times[0, 2 \pi] \to \CP^3 : (t, \alpha, \beta) \mapsto \hpi( e^{t M(\alpha, \beta) } p_0 ), \qquad p_0 = \mqty( 1 & 0 & 0 & 0 )^T.
    \end{equation}
    We call this the \emph{exceptional homogeneous Lagrangian}. 
    
    Being the orbit of a subgroup of $\Sp(2)$, it is extrinsically homogeneous, and thus we can pick one point to show that it is Lagrangian and to compute the angle. Let $F(t, \alpha, \beta) = e^{t M(\alpha, \beta) } p_0$. We choose the following point, as it makes the computations easy: $q = e^{\pi/2 M(\pi/2,0)} p_0$. At $q$, the vectors $\partial_t F$, $\partial_\alpha F$ and $\partial_\beta F +7/3 i F$ are horizontal and orthogonal. The last one lies in $\lift{\Dfour}$. Normalising these vectors with respect to $g$, we can compute that the immersion is indeed Lagrangian, and that the angle is constant, namely $\theta = \arctan{\frac{7}{\sqrt{76 + 15 \sqrt{15}}}}$.

    We proceed to give some properties of this Lagrangian. Let $( e_1, e_2, e_3)$ be the canonical frame of \Cref{thm:intro_lagframe_ABform} on this Lagrangian. With the Gauss equation, we compute the following sectional curvatures: 
    \[ \sec_\mathcal{L}( e_1, e_2) = \frac{9}{100} \l( 6 \sqrt{15} - 11 \r), \quad  \sec_\mathcal{L}( e_1, e_3) = -\frac{9}{100} \l( 6 \sqrt{15} + 11 \r), \quad \sec_\mathcal{L}( e_2, e_3) = \frac{207}{100}. \]
    In terms of these components, we find that for all orthonormal $X,Y \in \mathfrak{X}(\mathcal{L})$:
    \[ \sec_\mathcal{L}(X,Y) = \frac{1}{2} \sum_{\substack{i,j=1\\ i\neq j}}^3 \sec_\mathcal{L}( e_i, e_j) \, g( G( e_i, e_j), G(X,Y) )^2. \]
    Computing the Ricci curvature directly from the definition, we find that it is diagonal in this frame and given by 
    \[ \Ric( e_1, e_1) = -\frac{99}{50}, \quad \Ric( e_2, e_2) = \frac{27}{50} \l(\sqrt{15}+2\r), \quad \text{ and } \quad  \Ric( e_3, e_3) = -\frac{27}{50} \l(\sqrt{15}-2\r). \]
    Finally, using the Gramm-Schmidt procedure, we can make an orthonormal frame for the metric $\FSm$, with which we can compute the mean curvature of this submanifold in the Kähler $\CP^3$. It is given by $\FS{H} = \frac{224}{61}\sqrt{\frac{2}{5}} \J  e_1$, and hence lies in the direction of $J$ applied to the canonical vector field $e_1$. 
\end{example}

There are a priori many ways to immerse the domains of the immersions of \Cref{eq:berger,eq:EHL,eq:Chiang,eq:RP3,eq:S2xS1} into $\CP^3$. However, we show that the above immersions are rigid in the following sense.

\begin{theorem}\label{thm:rigidity}
Let $f: (M, \langle \cdot,\cdot \rangle_g) \to (\CP^3,g)$ be any of the isometric immersions of \Cref{eq:berger,eq:EHL,eq:Chiang,eq:RP3,eq:S2xS1}, where $\langle \cdot,\cdot \rangle_g$ denotes the pull back metric on $M$ under $f$ of the nearly Kähler metric $g$ on $\CP^3$. For any other isometric immersion $f': (M, \langle \cdot,\cdot \rangle_g) \to \CP^3$, there is an isometry $\phi$ of the nearly Kähler $\CP^3$ such that $f' = \phi \circ f$.
\end{theorem}

To prove this, we first recall the following result, which is taken from \cite{reckziegel1981a} (Proposition 3). 

\begin{proposition}[\cite{reckziegel1981a}]\label{thm:tool_congruence_from_frame_congruences}
Let $M$ be an $n$-dimensional Riemannian manifold and let $I$ be an open interval containing $0$ and let $\gamma,\tilde{\gamma}:I \to M$ be differentiable curves such that $\gamma(0) = \tilde{\gamma}(0)$. Suppose that $(E_1, \hdots, E_n)$ and $(\tilde{E}_1,\hdots,\tilde{E}_n)$ are frames along $\gamma$ and $\tilde{\gamma}$, respectively, that coincide at $0$. If $\inprod{\gamma'}{E_k} = \inprod{\tilde{\gamma}'}{\tilde{E}_k}$ and $\inprod{\nabla_{\gamma'} E_k}{E_l} = \inprod{\nabla_{\tilde{\gamma}'} \tilde{E}_k}{\tilde{E}_l}$ for all $k,l \in \set{1,\hdots,n}$, then, $\gamma=\tilde{\gamma}$ and $E_k = \tilde{E}_k$ for all $k \in \set{1,\hdots,n}$. 
\end{proposition}

\begin{proof}[Proof of \Cref{thm:rigidity}]
Since the Veronese surface and the totally geodesic two-sphere are both rigid surfaces in $\sphere{4}$, and the correspondence of \cite{storm2020} is a bijection, rigidity of these examples follows. 

We do the proof for the Berger sphere and for $\sphere{2}\times \sphere{1}$ explicitly; the proof for the exceptional Lagrangian is analogous to that for the Berger sphere. We first consider the two examples separately to construct frames on the submanifolds. Then we continue the argument for both at the same time.

\textit{Step 1: a frame on $\sphere{2} \times \sphere{1}$.} Let $M$ denote $\sphere{2} \times \sphere{1}$. Take any point $p_0$ of $M$ and note that by homogeneity of the nearly Kähler $\CP^3$, we can assume that $f(p_0) = f'(p_0)$. Consider a unit length vector field $E_1 = \pdv{t}$ on $\sphere{1}$. Choose any unit length vector field $E_2$ tangent to $\sphere{2}$, which is automatically normal to $E_1$ and consider $E_3=E_1 \cross E_2$, where $\cross$ is the cross product on $\reals^3 \cong T(\sphere{1}\times \sphere{2})$. One can now build frames $(A_1,\ldots,A_6)$ and $(B_1,\ldots,B_6)$ along the images of $f$ and $f'$ by $A_i = (\dd f) E_i$, $A_{i+3} = \J A_i$, $B_i = (\dd f') E_i$ and $B_{i+3} = \J B_i$ for $i\in\{1,2,3\}$.

\textit{Step 2: a frame on the Berger sphere}. Let $M$ denote the Berger sphere. Take any point $p_0$ of $M$ and assume that $f(p_0) = f'(p_0)$, as before. Associated with both immersions, we can build the frame (perhaps on a neighbourhood of $f(p_0)$) of \Cref{thm:intro_lagframe_ABform}, say $\set{  e_1 , e_2 , e_3 }$ and $\set{  e_1' , e_2' , e_3' }$, which span $(\dd f_p)(T_{f(p)} M) \subset T_{f(p)} \CP^3$ and $(\dd f'_p)(T_{f'(p)} M) \subset T_{f'(p)} \CP^3$, respectively. 

Since $\theta$ is equal to $\pi/4$, we can rotate $e_2$ and $e_3$ (or $e_2'$, $e_3'$). By doing this over an appropriate angle (satisfying the integrability conditions), we find that the $\omega_{12}^3, \omega_{22}^3, \omega_{33}^2$ of \Cref{thm:lagr_pi/4_possibilities} are constant and given by $\omega_{12}^3 = -9/4$ and $\omega_{22}^3 = \omega_{33}^2 = 0$. Moreover, since $\dd f$ is an isomorphism between $T M$ and  $\dd f_p \l(T_{f(p)} M\r)$, and similarly for $f'$, we have a frame $\set{E_1, E_2, E_3}$ on $M$ that maps to $\set{ e_1, e_2, e_3}$. Similarly, we have $\set{E_1', E_2', E_3'}$ on $M$ coming from $f'$. As the connection coefficients in these frames are constant and equal, \Cref{thm:constant_connection_coef_ONF_submanifold} tells us that there is a local isometry $\phi$ such that $\phi(p_0)=p_0$ and $\dd \phi E_i = E_i'$ on a neighbourhood of $p_0$. Hence, without loss of generality, we can assume that we are working on this neighbourhood of $p_0$ and with the frame $\set{E_1,E_2,E_3}$ on $M$. Transporting this frame to $\CP^3$ via the two immersions, we can construct frames $\set{A_i}$ and $\set{B_i}$.

\textit{Step 3: finishing the proof with a shared argument.} With the frames $(A_1,\ldots,A_6)$ and $(B_1,\ldots,B_6)$, we can continue in general. Note that we can assume that $A_i(f(p_0)) = B_i(f(p_0))$, 
again after applying an isometry of the nearly Kähler $\CP^3$. Indeed, note that $A_1,B_1 \in \Dfour$ and 
\begin{align*}
A_2 &= \cos \theta \, \Phi A_1+ \sin \theta \, U, &
A_3 &= \cos \theta \, \J \Phi A_1 - \cos \theta \, \J U, \\
B_2 &= \cos \theta \Phi \, B_1+ \sin \theta \, U', & 
B_3 &= \cos \theta \J \Phi \, B_1 - \cos \theta \, \J U',
\end{align*}
with $U,U' \in \Dtwo$. Hence, we can apply isometries $F_1$ and $F_2$ of the nearly Kähler $\CP^3$ so that $p_0$ is fixed by both $F_1$ and $F_2$ and so that $\dd (F_1^{-1} \circ F_2)\, U'(p_0) = U(p_0)$ and $\dd (F_1^{-1} \circ F_2) \, B_1(p_0) = A_1(p_0)$. Thus, we now assume that $A_i(p_0)=B_i(p_0)$ for $i \in \set{1, \hdots, 6}$. 

Next, given any geodesic $\gamma:(-\epsilon,\epsilon) \to M$ with $\gamma(0) = p_0$ and consider $\alpha = f\circ \gamma $ and $\beta=f'\circ \gamma$. Note that $\alpha(0)=\beta(0)=f(p_0)=f'(p_0)$. 
Since both $f$ and $f'$ are isometric immersions, we have $g(\alpha', A_i) = g(\dd f \gamma', \dd f E_i) = \inprod{\gamma'}{E_i}_g = g(\dd f' \gamma', \dd f' E_i) = g(\beta', B_i)$ for $i \in \set{1,2,3}$. Moreover, since $\alpha'$ and $\beta'$ are tangent to $M$, we have $g(\alpha', A_i) = 0 = g(\beta', B_i)$ for $i \in \set{4,5,6}$.

Finally, we prove that $g(\nabla_{\alpha'}A_k, A_l)=g(\nabla_{\beta'}B_k, B_l)$ for all $k,l \in \set{1,\hdots,6}$. To do so, we decompose $\nabla$ into its tangent and normal parts, thus introducing the second fundamental forms: $\nabla_X Y = \dd f (\nabla^M_X Y) + h_f(X,Y)$ and $\nabla_X Y = \dd f' (\nabla^M_X Y) + h_{f'}(X,Y)$. From \Cref{thm:lagr_pi/4_possibilities}, we know $\inprod{\nabla^M_{\gamma'} E_i}{E_j}_g$, since $\gamma' = a_1 E_1+a_2 E_2+a_3 E_3$. Furthermore, this theorem also gives the components of $\J h$ in terms of the parameters involved in the connection coefficients. It follows that $g(\nabla_{\alpha'}A_k, A_l) = \inprod{\nabla^M_{\gamma'}E_k}{E_l}_g = g(\nabla_{\beta'}B_k, B_l)$ for $k,l \in \set{1,2,3}$, and 
\begin{align*}
&g( \nabla_{\alpha'}A_k, A_l )
= g( \nabla_{\alpha'}A_k, \J A_{l-3} ) = g( h_f(\alpha',A_k), \J A_{l-3} ) \\
&= - g( \J h_f(\alpha',A_k), A_{l-3} ) 
= - \inprod{\J h(\gamma', e_k)}{ e_{l-3}}_g
= - g( \J h_{f'}(\beta',B_k), B_{l-3} ) \\
&= g( h_f(\beta',B_k), \J B_{l-3} ) = g( \nabla_{\beta'}B_k, \J B_{l-3} )
= g( \nabla_{\beta'}B_k, B_l )
\end{align*}
for $k\in \set{1,2,3}$ and $l\in \set{4,5,6}$. Since the $(A_1,\ldots,A_6)$ and $(B_1,\ldots,B_6)$ are orthonormal frames, we know $g(\nabla_{\alpha'}A_k, A_l) = -g( A_k,\nabla_{\alpha'}A_l)$ and $g(\nabla_{\alpha'}B_k, B_l) = -g( B_k,\nabla_{\alpha'}B_l)$, so that the case $k\in \set{4,5,6}$ and $l\in \set{1,2,3}$ is also proven. 

Hence, from \Cref{thm:tool_congruence_from_frame_congruences} we can conclude that $\alpha = \beta$ and that $A_i = B_i$ for all $i$. Since this was done for an arbitrary geodesic $\gamma$, we have the same result on a geodesic neighbourhood of $p_0$. 
As such, using only isometries of the nearly Kähler $\CP^3$, we reduced $f'$ to the given $f$, so that $f$ is indeed rigid. 
\end{proof}

\subsection{Classification of extrinsically homogeneous Lagrangians}

\begin{theorem}\label{thm:classification}
The Lagrangian immersions of \Cref{eq:RP3,eq:Chiang,eq:berger,eq:S2xS1,eq:EHL} are the unique extrinsically homogeneous Lagrangians in the nearly Kähler $\CP^3$. 
\end{theorem}

\begin{proof}
There are (at least) three ways to prove this statement. The first method is showing that an extrinsically homogeneous Lagrangian has constant angle, and other constant coefficients in the frame of \Cref{thm:intro_lagframe_ABform}. This approach can be found in detail in \cite{liefsoens2024thesis}. The second method is listing all subgroups of the identity component of the isometry group of the nearly Kähler $\CP^3$ (up to conjugation) and checking which of their orbits are Lagrangian. This is feasible as $\Sp(2)$ is a semi-simple compact Lie group. The third option, and the one we present here for conciseness, is combining the second option with a classification result of \cite{aslan2023}. 

Let $K$ be a Lie subgroup of $\Sp(2)$ with Lie algebra $\mathfrak{k}$. Denote by $H = \Sp(1) \times \U(1)$ the isotropy subgroup, so that $\Sp(2)/H = \CP^3$. Let $\mathfrak{h}$ be the Lie subalgebra of $H$. Finally, denote $\mathfrak{k}_0$ for the Lie algebra of $K \cap H$. For a Lagrangian orbit, we need $\dim(\mathfrak{k}) - \dim(\mathfrak{k}_0) = 3$. In \cite{kerr_nonnegatively_2013}, the Lie subalgebras of $\so(5) \cong \sp(2)$ are classified up to conjugacy. From this classification, we find that the only options for $(\mathfrak{k}, \mathfrak{k}_0)$ are $(\su(2), \set{0})$ and $(\su(2) \oplus \u(1) \cong \u(2), \u(1))$. For completeness: the former gives rise to the Chiang Lagrangian and the exceptional homogeneous Lagrangian, while the latter gives rise to the product $\sphere{2}\times \sphere{1}$, the Berger sphere and the real projective space. If $M$ is a Lagrangian orbit under one of the corresponding groups, it thus allows a non-trivial $\SU(2)$-action. Invoking the classification of \cite{aslan2023}, completes the proof.  
\end{proof}

\subsection{Lagrangian submanifolds with constant sectional curvature}\label{sec:cst_sec_curv}

In this section, we prove that there exist no Lagrangian submanifolds with constant sectional curvature in the nearly Kähler $\CP^3$. We will prove this by first proving that a Lagrangian submanifold with constant angle and constant sectional curvature cannot exist. For this, we need to prove separately the case $\theta$ constant and zero, and $\theta$ constant and non-zero. 
For $\theta= 0$, we need the following computational result.
\begin{lemma}\label{thm:h_omega_Lagr_B_zero}
Suppose we have a Lagrangian submanifold of the nearly Kähler $\CP^3$ with constant angle $\theta =0$. Let $(e_1,e_2,e_3)$ be a frame as in \Cref{thm:intro_lagframe_ABform}. Denote $h_{ij}^k = g(\nabla_{ e_i} e_i, \J  e_k)$ and $\omega_{ij}^k = g(\nabla_{ e_i} e_i,  e_k)$. Up to symmetry, the non-zero connection coefficients and second fundamental form components are linked as follows:
\begin{align*}
\omega_{11}^2 &= \lambda_1, & \omega_{21}^2 &= \lambda_2, &
\omega_{12}^3 = \omega_{23}^1 &= -\frac{1}{2}, & \omega_{31}^2 &= \lambda_3 -\frac{1}{2}, & h^1_{11} &=-h^1_{22} = \mu, & h^2_{11} &=-h^2_{22} = \rho.
\end{align*}
Here, $\mu$, $\rho$ and $\lambda_i$ are introduced, and they can be recognised as being $\mu = g(\FSnabla_{ e_1}  e_1 , \J e_1)$, $\rho = g(\FSnabla_{ e_1}  e_1 , \J e_2)$ and $\lambda_i = g(\FSnabla_{ e_i}  e_1 ,  e_2)$. 
\end{lemma}
\begin{proof}
    To link these coefficients, it takes a simple computation using \Cref{thm:link_nablaJ_nablaJ0}. 
\end{proof}
\begin{proposition}\label{thm:theta=0_cst_sec_curv}
A Lagrangian submanifold of the nearly Kähler $\CP^3$ with constant angle $\theta = 0$ cannot have constant sectional curvature.
\end{proposition}
\begin{proof}
Assume the notation and results of \Cref{thm:h_omega_Lagr_B_zero}. Calculating the sectional curvatures, we find $\sec_\mathcal{L}( e_1, e_2) = \frac{1}{4} +  e_2(\lambda_1) -  e_1(\lambda_2) -\lambda_1^2-\lambda_2^2-\lambda_3$, $\sec_\mathcal{L}( e_1, e_3) = \sec_\mathcal{L}( e_2, e_3) = \frac{1}{4}$. The Ricci equation states
$$ C_r(X,Y,\xi, \eta) := g(R^\perp(X,Y)\xi,\,\eta) - g(\tilde{R}(X,Y)\xi,\,\eta) - g( [A_\xi,A_\eta]X,\, Y) = 0. $$
Hence, $0 = C_r( e_1, e_2,\J  e_1, \J  e_2) =  e_1(\lambda_2) -  e_2(\lambda_1)+\lambda_1^2+\lambda_2^2+\lambda_3-2 \mu ^2-2 \rho ^2+1$. We add this to $\sec_\mathcal{L}( e_1, e_2)$ to obtain $\sec_\mathcal{L}( e_1, e_2) + C_r( e_1, e_2,\J  e_1, \J  e_2) = \frac{1}{4} + 1-2 \mu ^2-2 \rho ^2$. Then, $(\frac{1}{\sqrt{2}})^2 = \mu ^2 + \rho ^2$. Hence, there must be a function $\alpha$ such that $\mu = \frac{1}{\sqrt{2}}\, \cos\alpha$ and $\rho = \frac{1}{\sqrt{2}}\, \sin\alpha$.

We now look at the Codazzi equation, stating 
\[ C_c(X,Y,Z, \J W) := g(\tilde{R}(X,Y)Z - (\nabla^\perp_{X} h)(Y,Z) + (\nabla^\perp_{Y} h)(X,Z),\, \J W) = 0. \]
From $C_c( e_1, e_3, e_1,\J e_2)=0$ and $C_c( e_2, e_3, e_1,\J e_2)=0$, we conclude that $3\lambda_3 = 1 -  e_3(\alpha)$. Define $\psi := \alpha/3$ and rotate the frame as follows: $ e^\psi_1 = \cos\psi\,  e_1 + \sin\psi\,  e_2$, $ e^\psi_2 = -\sin\psi\,  e_1 + \cos\psi\,  e_2$, and $ e^\psi_3 =  e_3$. In the rotated frame, we get the following non-zero components of the connection coefficients and second fundamental form
\begin{gather*} 
3\, \omega_{11}^2 = \cos \psi\, (  e_1(\alpha) +3 \lambda_1 )+\sin \psi\, ( e_2(\alpha) + 3 \lambda_2 ), \\
3\, \omega_{21}^2 = \cos \psi\, (  e_2(\alpha) +3 \lambda_1 )-\sin \psi\, ( e_1(\alpha) + 3 \lambda_1 ), \\ 
3\, \omega_{31}^2 = \omega_{12}^3 = \omega_{23}^1 = -\frac{1}{2}, \qquad
h^1_{11} = -h^1_{22} = \frac{1}{\sqrt{2}}.
\end{gather*}

Again from the Codazzi equation, we find
\[ \sqrt{2} C_c( e^\psi_1, e^\psi_2, e^\psi_1,\J e^\psi_2) = \sin \psi\, ( e_1(\alpha) +3 \lambda_1)-\cos\psi\,( e_2(\alpha) +3 \lambda_2) = 0. \]
The most general solution to this equation is $\lambda_1 = \frac{1}{3} ( f \cos \psi\, -  e_1(\alpha) )$, and $\lambda_2 = \frac{1}{3} ( f \sin\psi -  e_2(\alpha) )$ for some function $f$. The connection coefficients in the frame then simplify to $\omega_{11}^2 = \frac{1}{3}\, f$, $\omega_{21}^2 = 0$ and $\omega_{12}^3 = \omega_{23}^1 = 3\, \omega_{31}^2 = -\frac{1}{2}$. As such, the remaining Ricci equations simplify as well. We find from $C_r( e^\psi_2, e^\psi_3,\J  e^\psi_1, \J  e^\psi_2)$ that $f=0$. From the final condition left from the Ricci equation, namely $0 = C_r( e^\psi_1, e^\psi_2,\J  e^\psi_1, \J  e^\psi_2) = \frac{1}{3}$, we then find a contradiction that finishes the proof.
\end{proof}

Further, we have the following two technical lemmas. We use the notation $\cyclic{}$ to denote a cyclic sum. 
\begin{lemma}\label{thm:nabla_perp_G_lagrangian}
Let $\mathcal{L}$ be a Lagrangian submanifold of the nearly Kähler $\CP^3$ with $\theta$ not vanishing on any neighbourhood. Denote the second fundamental form by $h$ and let $\hat{h}=-\J h$ and $G = \nabla \J$. Then, for all $W,X,Y,Z \in \mathfrak{X}(\mathcal{L})$, the following holds:
\[ \cyclic{W,X,Y} \Big( (\nabla^\perp G)(W,X,\hat{h}(Y,Z)) - (\nabla^\perp G)(X,W,\hat{h}(Y,Z)) \Big) = 0. \] 
\end{lemma}
\begin{proof}
We prove this for the frame of \Cref{thm:intro_lagframe_ABform}. We make only one explicit computation as the others are analogous. By anti-symmetry of $G$ and the property $G( e_i, e_j) = \epsilon_{ijk} \J  e_k$, we find
\begin{align*}
(\nabla^\perp G)( e_1, e_2,   e_1)
%
&= ( \nabla_{ e_1} ( -\J  e_3)  )^\perp - G( \omega_{12}^2  e_2 + \omega_{12}^3 e_3,  e_1) - \omega_{11}^3 G( e_2,  e_3) \\
%
%
&= - ( -\J  e_2 + \omega_{13}^1 \J  e_1 + \omega_{13}^2 \J  e_2 )
- \omega_{12}^3 \J  e_2
- \omega_{11}^3 \J  e_1 
%
%
= \J  e_2.
\end{align*}
Similarly, we have $(\nabla^\perp G)( e_1, e_2,  e_3)= 0$. Hence, we find
\begin{align*}
g( (\nabla^\perp G)( e_1, e_2,\hat{h}( e_3, e_3)), \J  e_1 ) 
g( (\nabla^\perp G)( e_1, e_2, h_{33}^1  e_1 + h_{33}^3  e_3), \J  e_1 )=0  
%
\end{align*}
Analogously, 
\begingroup
\allowdisplaybreaks
\begin{align*}
g( (\nabla^\perp G)( e_2, e_1,\hat{h}( e_3, e_3)), \J  e_1 ) 
= g( (\nabla^\perp G)( e_3, e_1,\hat{h}( e_2, e_3)), \J  e_1 )&=  e_3(\theta), \\
g( (\nabla^\perp G)( e_2, e_3,\hat{h}( e_1, e_3)), \J  e_1 ) 
= g( (\nabla^\perp G)( e_3, e_2,\hat{h}( e_1, e_3)), \J  e_1 )  &= 0, \\
g( (\nabla^\perp G)( e_1, e_3,\hat{h}( e_2, e_3)), \J  e_1 ) 
&= 0.
\end{align*}
\endgroup

For $W =  e_1$, $X =  e_2$ and $Y =  e_3$, we then find that the left hand side of the stated equation vanishes when projected to $\J  e_1$. Analogously, we can show that the projections to $J e_2$ and $J e_3$ also vanish, so that the left hand side vanishes identically. 
\end{proof}

\begin{lemma}\label{thm:conditions_2ndfundform}
Let $\mathcal{L}$ be a Lagrangian submanifold of the nearly Kähler $\CP^3$ with constant sectional curvature. Let $R$ denote the curvature tensor of the nearly Kähler $\CP^3$, and $h$ the second fundamental form of $\mathcal{L}$. The following equation holds for all $U,X,Y,Z,W \in \mathfrak{X}(\mathcal{L})$: 
\[ \cyclic{U,X,Y} g\l( R(X,Y) h(U,Z) - h(U, (R(X,Y)Z)^\top) , \J W \r) = 0. \]
\end{lemma}
\begin{proof}
By differentiating the Codazzi equation with respect to $U$, we obtain
\begin{align}\label{eq:first_step}
&g( (\nabla R)(U, X,Y,Z) + R(h(U,X), Y )Z
+ R(X, h(U,Y) )Z + R(X, Y )h(U,Z) , \J W)  \\ & \qquad+ g( R(X,Y)Z, G(U,W) - A_{\J W} U ) \nonumber \\
&=
g( (\nabla^2 h)(U,X,Y,Z) - (\nabla^2 h)(U,Y,X,Z), \J W) \nonumber \\
&\qquad + g( (\nabla h)(X,Y,Z) - (\nabla h)(Y,X,Z), G(U,W) ). \nonumber 
\end{align}

The Ricci identity, on the other hand, implies
\begin{align*}
\cyclic{U,X,Y}& \l( g( (\nabla^2 h)(U,X,Y,Z) - (\nabla^2 h)(U,Y,X,Z), \J W) \r) \\
&= \cyclic{U,X,Y} \l( g(   R^\perp(X,Y)h(U,Z) - h(R^\mathcal{L}(X,Y)U, Z )  -h(U, R^\mathcal{L}(X,Y)Z )  , \J W) \r).
\end{align*}
Applying \Cref{eq:R_perp_lagrangian} and Lemma \ref{thm:nabla_perp_G_lagrangian} to this equation yields
\begin{align*}
&\cyclic{U,X,Y} \l( g( (\nabla^2 h)(U,X,Y,Z) - (\nabla^2 h)(U,Y,X,Z), \J W) \r) \\
&= - \cyclic{U,X,Y} \l( g(  \J R^\mathcal{L}(X,Y)\J h(U,Z)
+ h(R^\mathcal{L}(X,Y)U, Z )  + h(U, R^\mathcal{L}(X,Y)Z )  , \J W) \r).
\end{align*}

Because of \Cref{thm:JG_fully_anti_symmetric_lagsub}, $G(U,W)$ is normal to $\mathcal{L}$. Hence, the Codazzi equation gives
\[ g( (\nabla h)(X,Y,Z) - (\nabla h)(Y,X,Z), G(U,W) ) = g( R(X,Y)Z , G(U,W) ). \]
Furthermore, the differential Bianchi identity states $\cyclic{U,X,Y} (\nabla R)(U, X,Y,Z) = 0$. If one combines all this information and assumes that $\mathcal{L}$ has constant sectional curvature $c$, the cyclic sum of \cref{eq:first_step} becomes
\begin{align*}
\cyclic{U,X,Y} & g\l( R(X,Y) h(U,Z) - h(U, (R(X,Y)Z)^\top) , \J W \r) \nonumber\\
=& 
- c\,  \cyclic{U,X,Y} g( \J (X \wedge Y) \J h(U,Z) +  h((X \wedge Y)U, Z )  +  h(U, (X \wedge Y)Z ), \J W).
\end{align*}
A direct computation shows that the right-hand side of this equation vanishes. We thus obtain the statement.
\end{proof}

The above two lemmas allow to extend \Cref{thm:theta=0_cst_sec_curv} to all Lagrangians with any constant angle function.
\begin{proposition}\label{thm:cst_sec_curv_cst_angle}
There does not exist a Lagrangian submanifold of the nearly Kähler $\CP^3$ with constant angle and constant sectional curvature.
\end{proposition}
\begin{proof}
Suppose $\mathcal{L}$ is a Lagrangian submanifold of the nearly Kähler $\CP^3$ with constant angle and constant sectional curvature. By \Cref{thm:theta=0_cst_sec_curv} and the classification of \Cref{thm:lagr_pi/4_possibilities}, we know that $\theta \not\in \set{0, \pi/4}$. We can then calculate in the frame of \Cref{thm:intro_lagframe_ABform} to get conditions on the components of the second fundamental form. 

Write $\sigma_i = g(\nabla_{ e_i}U, JU)$, with $U$ as in \Cref{thm:intro_lagframe_ABform}. From \Cref{thm:conditions_2ndfundform} it follows that $(4 - 5 \cos (2 \theta )) \sigma_3=0$, $(4 + 5 \cos(2 \theta ) ) \sigma_2 =0$ and $\frac{2 (5 \cos (4 \theta )-3)}{5 (1-\cos (4 \theta ))}\sigma_1 = \mu$. It is straightforward to show that $\cos(2\theta) = \pm\frac{4}{5}$ is not possible, as it gives a contradiction to the Gauss and Codazzi equations. Hence, $\sigma_2 =\sigma_3=0$.
	
Next, we consider the sectional curvature. From the Gauss equation, we find 
\begin{align*}
\sec_\mathcal{L}( e_1, e_2) &= \frac{1}{100} \l(\sigma_1 (\sigma_1 \csc ^2(\theta )-19 \sigma_1 \sec ^2(\theta )-10 \tan (\theta )+90 \cot (\theta
   ))+25 (4 \sigma_1^2+1)\r), \\
\sec_\mathcal{L}( e_1, e_3) &= \frac{1}{100} \l(\sigma_1 (-19 \sigma_1 \csc ^2(\theta )+\sigma_1 \sec ^2(\theta )+90 \tan (\theta )-10 \cot (\theta))+25 (4 \sigma_1^2+1)\r), \\
\sec_\mathcal{L}( e_2, e_3) &= \frac{1}{800} \csc ^2(\theta ) \sec ^2(\theta ) \l(25 (4 \sigma_1^2-1) \cos (4 \theta )+40 \sigma_1 \sin (2 \theta )+36 \sigma_1^2+25\r).
\end{align*}

By demanding that the difference between the first two equations is zero, we obtain $0 = 4 \sigma_1 \cot (2 \theta ) \csc (2 \theta ) \l(\frac{5}{2} \sin (2 \theta )+\sigma_1\r)$. From this, either $\sigma_1=0$ or $\sigma_1 = -\frac{5}{2} \sin (2 \theta )$. The former case results in a totally geodesic submanifold, and thus to the $\reals P^3$ of \Cref{eq:RP3}. As such, it cannot have constant sectional curvature. We finish the proof by discussing the latter possibility. 

In this case, considering $\sec_\mathcal{L}( e_1, e_2) - \sec_\mathcal{L}( e_2, e_3)$, we find that $25 \cos(4 \theta) = - 7$ and $\sigma_1 = -2$. This is again a very specific case, and we can easily check that it gives rise to a contradiction with the Ricci equation applied to $(e_1, e_2, e_1, e_3)$. This proves that also the last case cannot occur. 
\end{proof}

Finally, for arbitrary angle function, we will reduce to the case of constant angle function and prove the following
\begin{theorem}\label{thm:cst_sec_curv}
    There does not exist a Lagrangian submanifold of the nearly Kähler $\CP^3$ with constant sectional curvature.
\end{theorem}
\begin{proof}
    We prove this via contradiction. Suppose there is a Lagrangian submanifold $\mathcal{L}$ with constant sectional curvature $c$. So, $R^\mathcal{L}(X,Y)Z = c (X\wedge Y)Z$. The Gauss equation then reads
    \[ g(R(W,X)Y,Z) = c g( (W\wedge X)Y,Z)- g(h(W,Z),h(X,Y)) + g(h(W,Y),h(X,Z)). \]
    Consider the frame of \Cref{thm:intro_lagframe_ABform}. There are nine independent equations coming from this form of the Gauss equation, and they can be obtained by choosing $(W,X,Y,Z) = (e_i,e_j,e_k,e_l)$ with 
    $(i,j,k,l) \in \set{ (a, b, 1, 2), (a, b, 1, 3), (a, b, 2, 3) \mid (a,b)=(1,2),(1,3),(2,3)}$. From these equations, we solve for the following derivatives of components of the second fundamental form: $e_1(h_{22}^2)$, $e_2(h_{11}^3)$, $e_2(h_{11}^2)$, $e_2(h_{12}^3)$ and $e_3(h_{11}^3)$, $e_3(h_{12}^3)$, $e_3(h_{11}^1)$, $e_3(h_{11}^2)$, $e_3(h_{22}^2)$. 

    Next, we consider the Gauss equation with $R^\mathcal{L}$ computed via the definition and $\nabla^\mathcal{L}$ computed in terms of the connection coefficients. Because of the previous step, these equations do not contain any derivatives and we can find $h_{12}^3 = 0$. 

    The next step consists of imposing the condition of \Cref{thm:conditions_2ndfundform}. With $(X,Y,Z,U,W) = (e_1, e_2 , e_a, e_3, e_b)$ for $(a,b) = (2,3),(1,3),(1,2)$, we can solve for $h_{11}^1$, $h_{11}^2$, $h_{11}^3$, respectively. These components are in terms of $\theta$ and $h_{22}^1$, $h_{22}^2$, $h_{22}^3$. 
    With these expressions, we can again consider the remaining equations that come from the Gauss equation with $R^\mathcal{L}$ computed by its definition. We can solve for the square of $h_{22}^1$ and of $h_{22}^3$. 
    From these, the derivatives of $h_{22}^1$ and $h_{22}^3$ follow. 
    Then, the gauss equation applied to $(e_1, e_3, e_1, e_3)$ gives $h_{22}^1$.
    Recall that we already had an expression for the square of that component, and these two expressions should be compatible. However, matching these equations, we get a quadratic equation in $c$ with coefficients only depending on $\theta$. We can solve this for $c$, which has to be constant, and find that then $\theta$ has to be constant as well, which is a contradiction by \Cref{thm:cst_sec_curv_cst_angle}.
    %
    %
    %
    %
    %
    %
    %
    %
\end{proof}

\pdfbookmark{Bibliography}{Bibliography}
\bibliographystyle{amsplain}
\bibliography{Bibliography}{}


\end{document}